\documentclass[11pt,reqno,tbtags,a4paper]{amsart}
\usepackage{amssymb}
\usepackage{mathabx}  % \widecheck
\usepackage{xpunctuate}
\usepackage{url}
\usepackage[square,numbers]{natbib}
\bibpunct[, ]{[}{]}{;}{n}{,}{,}

\title%[Descendants in a random directed acyclic graph]
{The number of descendants in a random directed acyclic graph}

\date{27 February, 2023}
\author{Svante Janson}
\thanks{Supported by the Knut and Alice Wallenberg Foundation}
\address{Department of Mathematics, Uppsala University, PO Box 480,
SE-751~06 Uppsala, Sweden}
\email{svante.janson@math.uu.se}
\newcommand\urladdrx[1]{{\urladdr{\def~{{\tiny$\sim$}}#1}}}
\urladdrx{http://www2.math.uu.se/~svante/}
\numberwithin{equation}{section}

\renewcommand\le{\leqslant}
\renewcommand\ge{\geqslant}

\allowdisplaybreaks

\theoremstyle{plain}% default
\newtheorem{theorem}{Theorem}[section]
\newtheorem{lemma}[theorem]{Lemma}

\theoremstyle{definition}

\newcommand\xqed[1]{%
    \leavevmode\unskip\penalty9999 \hbox{}\nobreak\hfill
    \quad\hbox{#1}}

\newtheorem{exampleqqq}[theorem]{Example}
\newenvironment{example}{\begin{exampleqqq}}
  {\xqed{$\triangle$}\end{exampleqqq}}
\newtheorem{remarkqqq}[theorem]{Remark}
\newenvironment{remark}{\begin{remarkqqq}}
  {\xqed{$\triangle$}\end{remarkqqq}}
\newtheorem{problem}[theorem]{Problem}

\theoremstyle{remark}

\newenvironment{romenumerate}[1][-10pt]{% optional argument changes indentation
\addtolength{\leftmargini}{#1}\begin{enumerate}% gives (i), (ii) etc.
 \renewcommand{\labelenumi}{\textup{(\roman{enumi})}}%
 }{\end{enumerate}}

\newcounter{oldenumi}
{\setcounter{oldenumi}{\value{enumi}}
\begin{romenumerate} \setcounter{enumi}{\value{oldenumi}}}
{\end{romenumerate}}

\newcounter{thmenumerate}

\newcounter{xenumerate}   %no left indentation; thus wider lines

\newcounter{steps}
\newcommand\stepx[1]{\smallskip\noindent\refstepcounter{steps}%
 \emph{Step \arabic{steps}: #1}\noindent}

\newcommand{\refT}[1]{Theorem~\ref{#1}}
\newcommand{\refTs}[1]{Theorems~\ref{#1}}

\newcommand{\refL}[1]{Lemma~\ref{#1}}
\newcommand{\refLs}[1]{Lemmas~\ref{#1}}
\newcommand{\refR}[1]{Remark~\ref{#1}}

\newcommand{\refS}[1]{Section~\ref{#1}}
\newcommand{\refSs}[1]{Sections~\ref{#1}}
\newcommand{\refSS}[1]{Section~\ref{#1}}
\begingroup
  \divide\count255 by 60
  \multiply\count255 by -60
  \advance\count255 by \time
  \ifnum \count255 < 10 \xdef\klockan{\the\count1.0\the\count255}
  \else\xdef\klockan{\the\count1.\the\count255}\fi
\endgroup

\newcommand\set[1]{\ensuremath{\{#1\}}}

\newcommand\xpar[1]{(#1)}
\newcommand\bigpar[1]{\bigl(#1\bigr)}
\newcommand\Bigpar[1]{\Bigl(#1\Bigr)}
\newcommand\biggpar[1]{\biggl(#1\biggr)}
\newcommand\lrpar[1]{\left(#1\right)}
\newcommand\bigsqpar[1]{\bigl[#1\bigr]}
\newcommand\sqpar[1]{[#1]}

\newcommand\lrsqpar[1]{\left[#1\right]}
\newcommand\cpar[1]{\{#1\}}

\newcommand\abs[1]{\lvert#1\rvert}
\newcommand\bigabs[1]{\bigl\lvert#1\bigr\rvert}
\newcommand\Bigabs[1]{\Bigl\lvert#1\Bigr\rvert}
\newcommand\biggabs[1]{\biggl\lvert#1\biggr\rvert}
\newcommand\lrabs[1]{\left\lvert#1\right\rvert}
\def\rompar(#1){\textup(#1\textup)}    % usage: \rompar(...)

\newcommand\Bigparfrac[2]{\Bigpar{\frac{#1}{#2}}}

\def\xexp(#1){e^{#1}}
\newcommand\ceil[1]{\lceil#1\rceil}

\newcommand\floor[1]{\lfloor#1\rfloor}

\newcommand\ntoo{\ensuremath{{n\to\infty}}}

\newcommand\ttoo{\ensuremath{{t\to\infty}}}

\newcommand\downto{\searrow}
\newcommand\upto{\nearrow}
\newcommand\punkt{\xperiod}    % xpunctuate
    
\newcommand\ie{i.e\punkt}
\newcommand\eg{e.g\punkt}

\newcommand\cf{cf\punkt}
\newcommand{\as}{a.s\punkt}

\newcommand\whp{w.h.p\punkt}

\newcommand{\tend}{\longrightarrow}
\newcommand\dto{\overset{\mathrm{d}}{\tend}}
\newcommand\pto{\overset{\mathrm{p}}{\tend}}
\newcommand\asto{\overset{\mathrm{a.s.}}{\tend}}
\newcommand\lito{\overset{L^1}{\tend}}
\newcommand\plito{\pto}
\newcommand\eqd{\overset{\mathrm{d}}{=}}

\newcommand\op{o_{\mathrm p}}

\newcommand\bbR{\mathbb R}

\newcounter{CC}
\newcounter{cc}
\newcommand\E{\operatorname{\mathbb E}{}} %better spacing this way??
\renewcommand\P{\operatorname{\mathbb P{}}}

\newcommand\Var{\operatorname{Var}}
\newcommand\Cov{\operatorname{Cov}}

\newcommand\Exp{\operatorname{Exp}}

\newcommand\Bin{\operatorname{Bin}}

\newcommand\Ge{\operatorname{Ge}}
\newcommand\NegBin{\operatorname{NegBin}}

\newcommand\rise[1]{^{\overline{#1}}}

\newcommand\gd{\delta}
\newcommand\gD{\Delta}

\newcommand\gam{\gamma}
\newcommand\gG{\Gamma}

\newcommand\gs{\sigma}

\newcommand\gu{\upsilon}
\newcommand\gU{\Upsilon}

\renewcommand\phi{\xxx}  %% WARNING

\newcommand\cA{\mathcal A}
\newcommand\cB{\mathcal B}
\newcommand\cC{\mathcal C}

\newcommand\cF{\mathcal F}

\newcommand\cI{\mathcal I}

\newcommand\cY{{\mathcal Y}}

\newcommand\indic[1]{\boldsymbol1\cpar{#1}}

\newcommand\qw{^{-1}}
\newcommand\qww{^{-2}}
\newcommand\qq{^{1/2}}
\newcommand\qqw{^{-1/2}}

\newcommand\intoo{\int_0^\infty}

\newcommand\oi{\ensuremath{[0,1]}}
\newcommand\ooi{(0,1]}

\newcommand\ooo{[0,\infty)}
\newcommand\xooo{(0,\infty)}

\newcommand\dd{\,\mathrm{d}}
\newcommand\ddx{\mathrm{d}}

\newcommand{\pgf}{probability generating function}

\newcommand{\gsf}{$\gs$-field}

\newcommand\lhs{left-hand side}
\newcommand\rhs{right-hand side}

\newcommand\Uoi{\mathsf U(0,1)}

\newcommand\nn{^{(n)}}

\newcommand\nni{^{(n+1)}}

\newcommand\et{e^{-t}}

\newcommand\cYY{{\mathcal Y}'}
\newcommand\hcY{\widehat{\mathcal Y}}
\newcommand\hxi{\hat\xi}
\newcommand\nux[1]{{\nu_{#1}}}
\newcommand\DD{\widehat D}
\newcommand\hhD{{\widehat D}'}
\newcommand\llnn{(\log n)^2/n}
\newcommand\hA{\widehat A}
\newcommand\hB{\widehat B}
\newcommand\ddt{\frac{\ddx}{\dd t}}
\newcommand\Mx{M^*}
\newcommand\xdd{^{(d-1)/d}}
\newcommand\xddw{^{-(d-1)/d}}
\newcommand\nqq{\sqrt{n}}

\newcommand\nm{_{n,m}}
\begin{document}

\begin{abstract} 
We consider a well known model of
random directed acyclic graphs of order $n$, obtained by recursively
adding vertices, where each new vertex has a fixed outdegree $d\ge2$ and the
endpoints of the $d$ edges from it are chosen uniformly at random among
previously existing vertices.

Our main results concern the number $X\nn$
of vertices that are descendants of $n$.
We show that $X\nn/n\xdd$ converges in distribution;
the limit distribution is, up to a constant factor, given by the $d$th root
of a Gamma distributed variable. $\gG(d/(d-1))$. 
When $d=2$, the limit distribution can also
be described as a chi distribution $\chi(4)$.
We also show convergence of moments, and find thus the asymptotics of the
mean and higher moments.
\end{abstract}

\maketitle

\section{Introduction}\label{S:intro}

A \emph{dag} is a directed acyclic (multi)graph, and a \emph{$d$-dag} is a dag
where one or several vertices are \emph{roots} with outdegree 0, and all
other vertices have outdegrees $d$.
(Here, $d$ is a positive integer; we assume below $d\ge2$.)

We consider, as many before us,
the random $d$-dag $D_n$ on $n$ vertices 
constructed recursively by starting with a single
root 1, and then adding vertices $2,3,\dots,n$ one by one, giving each new
vertex, $k$ say, $d$ outgoing edges with endpoints uniformly and
independently chosen at random among the already existing vertices
$\set{1,\dots,k-1}$. (We thus allow multiple edges, so $D_n$ is a
directed multigraph.)
Two minor variations that will be discussed in \refS{Svar} are that we may
start with any number $m\ge1$ of roots, and that we may select the $d$
parents of a new node without replacement, thus not allowing multiple edges.
(In the latter case, we have to start with $\ge d$ roots.)

Note that for $d=1$, the model becomes the well known
\emph{random recursive tree}; the
properties in this case are quite different from 
the  case $d\ge2$, and we
assume throughout the paper $d\ge2$. In fact, to concentrate on the essential
features, in the bulk of the paper we consider the most important case $d=2$;
the minor differences in the case $d>2$ are briefly treated in \refS{Sd}.

The random $d$-dag has been studied as a model for a
random circuit where each gate has $d$ inputs chosen at random
\cite{Diaz+,TX,Arya+,TM,BroutinF,Moler+}.
(In this case it seems more natural to reverse all edges, and regard a
$d$-dag as a graph with indegrees 0 or $d$. 
 In the present paper, we direct the edges towards the
root(s) as above.)
The model has also been studied in
connection with constraint satisfaction
\cite[Exercise 7.2.2.3--371]{Knuth7A}. %page 78
Among  results shown earlier for random $d$-dags, we mention
results on vertex degrees and leaves
\cite{DevroyeLu, 
TM, %(leaves = indegree 0)
MT,
Moler+,
KubaS}, and 
on  lengths of paths and depth
\cite{Diaz+,
TX,
Arya+,
SJ235,
BroutinF}. 

In the present paper, we study the following problem, as far as we know
first considered by Knuth
\cite[Exercises 7.2.2.3--371 and 372]{Knuth7A}: %, page 78
How many descendants does vertex $n$ have?
In other words,
how many vertices can be reached by a directed path from vertex $n$?
In the random circuit interpretation, this is the number of gates (and
inputs) that are used in the calculation of an output.

We state our main results in the next subsection, and prove them in
\refSs{Sbasic}--\ref{Sd}.
Along the way, we prove some results on the structure of 
the subgraph of descendants %$\DD_n$,
which may be of independent interest.
Some further results are given in \refS{Sfurther}.
As said above, we discuss two variations of the model in \refS{Svar}.

\begin{remark}
 We emphasise that we 
in this paper exclusively consider random dags constructed by
\emph{uniform} attachment.
Another popular model that has been studied by many authors 
(often as an undirected graph)
is preferential
attachment, see \eg{}
\cite{BA} and \cite{BRST}.
A different model of non-uniform attachment is studied in \cite{BroutinF}.
\end{remark}

\begin{problem}
Find results  
for preferential attachment random dags
corresponding to the results above!
\\
Do the same for the model in \cite{BroutinF}!
\end{problem}

\subsection{Main result}\label{SSmain}
We introduce some notation; for further (mainly standard) notation, see
\refSS{SSnot}. We let $d\ge2$ be fixed and consider asymptotics as \ntoo.

Let $D_n$ be the random $d$-dag defined above, 
let $\DD_n$ be the subdigraph of $D_n$ consisting of all vertices and edges
that can be reached by a directed path from vertex $n$ (including vertex $n$
itself),
and let $X\nn:=|\DD_n|$, the number of descendants of $n$.
We thus want to find the asymptotic behaviour of the 
random variable $X\nn$ and its
expectation $\E X\nn$ as \ntoo.
Note that $\DD_n$ also is a $d$-dag, and has 1 root; 
thus the number of edges in $\DD_n$ is $d(X\nn-1)$, and hence 
our results also yield the asymptotics
of the number of edges.

Our main result in the case $d=2$ is the
following theorem, proved in two parts in \refSs{Spf1} and \ref{Smom}.

Let $\chi_4$ denote a random variable with the $\chi(4)$ distribution.
Recall that this means that $\chi_4$ has the distribution of $|\eta|$ where
$\eta$ is a standard normal random vector in $\bbR^4$, and that thus
(or by \eqref{Gamma} and a change of variables)
$\chi_4$ has density function
\begin{align}\label{chi1}
  f_{\chi_4}(x)=
  \frac12 x^3 e^{-x^2/2},
\qquad x>0.
\end{align}

\begin{theorem}\label{TX}
Let $d=2$.
Then,  as \ntoo,
  \begin{align}\label{tx1}
    X\nn/\sqrt n \dto
\frac{\pi}{2\sqrt2}\chi_4
  \end{align}
with convergence of all moments.
Hence, for every fixed  $r>0$,
\begin{align}\label{tx2}
  \E(X\nn)^r \sim \Bigparfrac{\pi}{2}^r\gG\Bigpar{\frac{r}{2}+2} n^{r/2}
\end{align}
and, in particular,
\begin{align}\label{tx3}
  \E X\nn \sim \frac{3\pi^{3/2}}{8} \sqrt n.
\end{align}
\end{theorem}

More generally, for any fixed $d\ge2$, we prove in \refS{Sd} the following:

\begin{theorem}\label{TD}
Let $d\ge2$. Then, as \ntoo,
  \begin{align}\label{td1}
    X\nn/n\xdd \dto 
\frac{\pi(d-1)^{1/d}}{d\sin(\pi/d)}\gam^{1/d},  
\end{align}
with convergence of all moments,
where $\gam\in\gG\bigpar{\frac{d}{d-1}}$.
Hence, for every fixed $r>0$,
\begin{align}\label{tdmom}
  \E\bigpar{X\nn}^r \sim
\lrpar{\frac{(d-1)^{1/d}\pi}{d\sin(\pi/d)}}^r 
\frac{\gG\bigpar{\frac{d}{d-1}+\frac{r}{d}}}{\gG\bigpar{\frac{d}{d-1}}}
 n^{r(d-1)/d}
\end{align}
\end{theorem}

We note that the convergence in \eqref{tx1} and \eqref{td1} does \emph{not}
hold a.s.; see \refR{Rnoas}. 

We will see in \refS{Svar} that the same results hold for the variations
with $m\ge1$ roots (as long as $m$ is fixed or does not grow too fast)
and without multiple edges (i.e., drawing without replacement).

\begin{example}
  Knuth \cite[Answer 7.2.2.3--371(b)]{Knuth7A} % page 140
considers the version with $d=2$, $m\ge2$ roots, and drawing without
replacement (\ie, no multiple edges); for this version he provides recursion
formulas that yield the exact value of $\E X\nn$ (there denoted $C_{m,n}$).
For example, for $m=2$ and $n=100$, his formulas yield 
$\E X\nn\doteq 20.79$ %(rounded) %20.78542838
while the asymptotic value \eqref{tx3} is 
$\doteq 20.88$, % 20.88122999
with an error of less than $0.5\%$.
\end{example}

\subsection{Notation}\label{SSnot}

%$C$ denotes unspecified constants that may vary from one occurrence to the
%next. 
%Sometimes we write \eg{} $C_r$ to emphasize that the constant depends on the
%parameter $r$.

The random $d$-dag $D_n$, its subdigraph $\DD_n$, and the number
$X\nn$ of descendants of $n$ are defined above.
The outdegree $d$ is fixed and not shown in the notation.
As said above we usually assume $d=2$; in particular this is the case in 
the proof in \refSs{Sbasic}--\ref{Smom}, while we consider general 
$d\ge2$ in \refS{Sd}.

We say that the vertices and edges of $\DD_n$ are \emph{red}.
Thus $X\nn:=|\DD_n|$ is 
the number of red vertices in $D_n$.
(For any digraph $D$, we let $|D|$ denote its number of vertices.)

Essentially all random variables below depend on $n$.
%, and usually some other parameter too. 
We may denote the dependency on $n$ by a supersript ${}\nn$
for clarity (in particular in limit statements), but we often omit this.
We sometimes in the proofs tacitly assume that $n$ is large enough.
Unspecified limits are as \ntoo.

We will in the proofs consider three different phases of the dag $D_n$,
see \refSs{SI}--\ref{SIII}. We will then  use fixed integers
$n_1=n_1\nn$ and $n_2=n_2\nn$; these can be chosen rather arbitrarily with
$n_1/n\to0$ slowly and $n_2/\sqrt n\to\infty$ slowly, see the beginnings of
\refSs{SI} and \ref{SII}.

We use $\pto$, $\dto$, $\lito$,
for convergence in probability, distribution and $L^1$, respectively, and
$\eqd$ for equality in distribution.

%$\Op$

As usual, \as{} (almost surely) means with probability 1, while
\whp{} (with high probability) means with probability tending to 1 as \ntoo.

We recall some classical probability distributions.
The Gamma distribution $\gG(s,a)$, where $s>0$ and $a>0$, 
has density 
\begin{align}\label{Gamma}
%\gG(s)\qw  x^{s-1}e^{-x},
\gG(s)\qw a^{-s} x^{s-1}e^{-x/a},
\qquad x>0.
\end{align}
We write $\gG(s)=\gG(s,1)$. 
(There should be no risk of confusion with the Gamma function.) 
In particular, $\gG(1)=\Exp(1)$, the standard exponential distribution. 
If $\gamma\in\gG(s,a)$, then \eqref{Gamma} implies
\begin{align}\label{Gmom}
  \E \gamma^r = a^r\frac{\gG(s+r)}{\gG(s)},
\qquad r\ge0.
\end{align}
The chi-square distribution $\chi^2(r)=\gG(r/2,2)$, and
the chi-distribution $\chi(r)$ is the distribution of $\sqrt{\xi}$ where
$\xi\in\chi^2(r)$.
(This shows that when $d=2$, the limits in \eqref{tx1} and \eqref{td1} agree.)

We use 'increasing' and 'positive' in the weak sense.

\begin{remark}
For simplicity, and to avoid unnecessary distraction, we 
often state results with convergence in probability, 
also when the proof yields the stronger
convergence in $L^1$.
(For example, this applies to all three results in \refS{SII}.)
Actually, in many (all?) cases, 
convergence in probability can be improved to
convergence in $L^p$ for any $p<\infty$, as a consequence of 
the estimates in \refS{Smom}.
\end{remark}

\begin{remark}\label{Rcoupling}
The construction of the random dag $D_n$ naturally constructs $D_n$ for all
$n\ge1$ together. In other words, it yields a coupling of $D_n$ for all
$n\ge1$.
However, in the proofs below we will \emph{not} use this coupling; instead we
regard $D_n$ as constructed separately for each $n$, which allows us to use
a different coupling in the proof.
\end{remark}

\section{Basic analysis}\label{Sbasic}
For simplicity, we assume $d=2$
from now on until the proof of \refT{TX} is completed at the end of
\refS{Smom}. The modifications for general $d$ are discussed in \refS{Sd}.

%\subsection{The red subdag}\label{SSred}
\subsection{A stochastic recursion}\label{SSrecursion}
We consider in the sequel only the red subgraph $\DD_n$ of $D_n$, which we
recall 
consists of the descendants of $n$ and
and all edges between them.
%; we call these vertices and edges \emph{red}. 
%Let $\DD_n$ denote the red subdigraph of $D_n$.

In the definition in \refS{S:intro} of the dag $D_n$, we start with vertex
1 and add vertices in increasing order. In our analysis, we will instead
start at vertex $n$ and go backwards to 1.
The red dag $\DD_n$ then
may be generated by the following procedure.

\begin{enumerate}
\item Start by declaring vertex $n$ to be \emph{red}, and all others
  \emph{black}. Let $k:=n$.
\item 
If vertex $k$ is red, then
create two new edges from that vertex, with
endpoints that are randomly drawn from $1,\dots,k-1$, and declare these
endpoints red.\\
If $k$ is black, delete $k$ (and do nothing else).
\item If $k=2$ then STOP; otherwise let $k:=k-1$ and REPEAT from (2).
\end{enumerate}

Let $Y_k$ be the number of  edges in $\DD_n$
that start in \set{k+1,\dots,n} and end in
\set{1,\dots,k}. In other words, $Y_k$ is the number of edges that cross the
gap between $k+1$ and $k$. 
Furthermore, let $Z_k$ be the number of these edges that end in $k$.
We here consider integers $k$ with $0\le k\le
n-1$, and have the boundary conditions $Y_{n-1}=2$ and $Y_0=0$; also
$Z_1=Y_1$ and $Z_0=0$.

Let also, for $1\le k\le n-1$,
\begin{align}\label{jk}
  J_k:=\indic{Z_k\ge1},
\end{align}
the indicator that at least one edge ends at $k$, which equals the
indicator that $k$ is red, and thus can be reached from $n$.

We will study the random dag $\DD_n$
by travelling from vertex $n$
backwards to the root; we thus consider the sequence $Y_{n-1},\dots,Y_1,Y_0$
in reverse order. In the procedure above, there are $Z_k$ edges that end at
$k$, and $2J_k$ edges that start there; hence,
for $2\le k\le n-1$,
\begin{align}\label{tw1}
  Y_{k-1}=Y_k-Z_k+2J_k=Y_k-Z_k+2\cdot\indic{Z_k\ge1}.
\end{align}

In our analysis,
we modify the procedure above by not revealing the endpoint of the edges
until needed. This means that when coming to a vertex
$k\in\set{1,\dots,n-1}$,
we have a list of $Y_k$ edges where we know only the start but not the end
(except that the end should be in \set{1,\dots,k}).
We then randomly select a subset by throwing a coin with
success probability $1/k$ for each of the $Y_k$ edges; these edges end at
$k$ and are removed from the list, and thus $Z_k$ is the number of them. 
This determines also $J_k$ by
\eqref{jk}, and if $J_k=1$, we add two new edges starting at $k$ to our list.
It is evident that this gives the same distribution of random edges as the
original algorithm above. 
(It is here important that the two edges from a
given vertex are chosen with replacement, so that we can treat the $Y_k$
edges passing over the gap between $k+1$ and $k$ as independent.
Note that the endpoints of these edges are uniformly distributed on
\set{1,\dots,k}.)

It follows from the modified procedure that $Y_{n-1},\dots,Y_1$
is a Markov chain. More precisely, 
let $\cF_k$ be the \gsf{} generated by our coin tosses at vertices
$n-1,\dots,k+1$, and note that  these coin tosses determine $Y_k$ (and also
$Y_{n-1},\dots,Y_{k+1}$). Then,
for $1\le k\le n-1$,
conditioned on $\cF_k$,
$Z_k$  has a binomial distribution
\begin{align}
  \label{zk}
Z_k\in\Bin(Y_k,1/k)
.\end{align}
Thus \eqref{tw1} and \eqref{zk} give a stochastic recursion of Markov type
for $Y_k$.

Note that $\cF_k \subset\cF_{k-1}$, so $\cF_1,\dots,\cF_{n-1}$ form a
\emph{decreasing} 
sequence of \gsf{s}, \ie, a revcerse filtration. 
We therefore may change sign of the indices and
consider, for example, $Y_{-j}$ and $\cF_{-j}$ for $j\in\set{-(n-1),\dots,-1}$
so that we have a filtration of the standard type.

The recursion \eqref{tw1}--\eqref{zk} yields, for $2\le k\le n-1$,
\begin{align}\label{tw2}
  \E \bigpar{Y_{k-1}\mid\cF_k}&
= Y_k -\E \bigpar{Z_k\mid \cF_k}+2\P\bigpar{Z_k\ge1\mid\cF_k}
\notag\\&
=Y_k-\tfrac{1}{k}Y_k+2\bigpar{1-(1-\tfrac1k)^{Y_k}}.
\end{align}
We obtain also, by Markov's inequality,
\begin{align}\label{tw2+}
  \E \bigpar{Y_{k-1}\mid\cF_k}&
\le Y_k -\E \bigpar{Z_k\mid \cF_k}+2\E \bigpar{Z_k\mid \cF_k}
%\notag\\&
=Y_k+\tfrac{1}{k}Y_k
=\tfrac{k+1}{k}Y_k
.\end{align}

\subsection{A reverse supermartingale and some estimates}
\label{SSmart}
We define, for $0\le k\le n-1$,
\begin{align}\label{tw3}
  W_k:=(k+1)Y_k,
\end{align}
and find from \eqref{tw2+}
\begin{align}\label{tw4+}
  \E \bigpar{W_{k-1}\mid\cF_k}&  
=k  \E \bigpar{Y_{k-1}\mid\cF_k}
\le (k+1)Y_k = W_k.
\end{align}
This shows that $W_{-j}$, $-(n-1) \le j\le 0$, is a supermartingale for the
filtration $(\cF_{-j})$; in other words, 
$W_0,\dots,W_{n-1}$ is a reverse supermartingale.
We have the initial value 
\begin{align}\label{tw40}
  W_{n-1}=nY_{n-1}=2n.
\end{align}
We thus have the Doob decomposition
\begin{align}\label{tm1}
  W_k=M_k-A_k,
\qquad 0\le k\le n-1,
\end{align}
where
\begin{align}\label{tm2}
  M_k&:=2n + \sum_{i=k+1}^{n-1}\bigpar{W_{i-1}-\E\xpar{W_{i-1}\mid\cF_i}}
\end{align}       
 is a reverse martingale: $\E\bigpar{M_{k-1}\mid\cF_k}=M_k$, and
%using \eqref{tw4=},
\begin{align}\label{tm3}
  A_k&:= \sum_{i=k+1}^{n-1}\bigpar{W_{i}-\E\xpar{W_{i-1}\mid\cF_i}}
%=\sum_{i=k+1}^{n-1}2i\bigpar{(1-\tfrac1i)^{Y_i}-1+\tfrac{Y_i}{i}}
%\ge0
\end{align}
is positive and reverse increasing: \eqref{tw4+} yields
\begin{align}\label{AAA}
0=A_{n-1}\le\dots\le A_1\le A_0.  
\end{align}
In particular, $W_k\le M_k$ and
\begin{align}\label{tm4}
  \E W_k \le \E M_k = M_{n-1}=2n,
\qquad 0\le k\le n-1.
\end{align}

We note also from \eqref{tw2} the exact formula
\begin{align}\label{tw4a}
    \E \bigpar{W_{k-1}\mid\cF_k}&
=k  \E \bigpar{Y_{k-1}\mid\cF_k}
%=(k-1)Y_k+2k\P\bigpar{Z_{k-1}\ge1\mid\cF_k}
=(k-1)Y_k+2k\bigpar{1-(1-\tfrac1k)^{Y_k}} 
\end{align}
and thus
\begin{align}\label{tw4b}
A_{k-1}-A_k&=
W_k- \E \bigpar{W_{k-1}\mid\cF_k}
=2Y_k-2k\bigpar{1-(1-\tfrac1k)^{Y_k}}
%\notag\\&
%=2k\bigsqpar{\E\bigpar{Z_{k-1}\mid\cF_k}-\P\bigpar{Z_{k-1}\ge1\mid\cF_k}}
\notag\\& 
=2k\bigsqpar{(1-\tfrac1k)^{Y_k}-1+\tfrac{Y_k}{k}}
.\end{align}

Furthermore, \eqref{tw1} and \eqref{zk} also yield
(rather crudely, but we do not need the exact formula, nor optimal constants),
for $1\le k\le n-1$,
\begin{align}\label{tw6}
  \Var\bigpar{Y_{k-1}\mid \cF_k}&
=   \Var\bigpar{Z_{k}-2\cdot\indic{Z_k\ge1}\mid \cF_k}
\notag\\&
\le 2\Var\bigpar{Z_{k}\mid \cF_k}
+ 2 \Var\bigpar{2\cdot\indic{Z_k\ge1}\mid \cF_k}
\notag\\&
= 2 Y_k\tfrac{1}{k}\bigpar{1-\tfrac{1}k}
+8 \P\xpar{Z_k\ge1\mid \cF_k}\bigpar{1-\P\xpar{Z_k\ge1\mid \cF_k}}
\notag\\&
\le 2 Y_k\tfrac{1}{k} 
+8 \P\xpar{Z_k\ge1\mid \cF_k}
\le 2 Y_k\tfrac{1}{k} 
+8 \E\xpar{Z_k\mid \cF_k}
\notag\\&
\le \tfrac{10}kY_k
\end{align}
and thus
\begin{align}\label{tw7}
  \Var\bigpar{W_{k-1}\mid \cF_k}&
=k^2\Var\bigpar{Y_{k-1}\mid \cF_k}
\le 10 k Y_k
%\notag\\&
\le 10 W_k.
\end{align}
Hence, \eqref{tm2} yields, using
the (reverse) martingale property, \eqref{tw7}, and
\eqref{tm4}, for $0\le k\le n-1$,
\begin{align}\label{tw8}
\Var M_k &=   \E\bigpar{M_k-2n}^2
=\sum_{i=k+1}^{n-1}\E\Var\bigpar{W_{i-1}\mid\cF_i}
\le10\sum_{i=k+1}^{n-1}\E W_{i}
\notag\\&
\le 10(n-1-k)\cdot 2n
%\notag\\&
\le 20n^2.
\end{align}
Consequently, since $W_k\le M_k$ by \eqref{tm1},
\begin{align}\label{tw9}
  \E W_k^2 \le \E M_k^2 =\Var M_k + (\E M_k)^2
\le 20n^2+(2n)^2=24n^2.
\end{align}

We extend this to a maximal inequality.
\begin{lemma}
  \label{LW*}
We have
\begin{align}\label{lw*}
  \E \max_{n-1\ge k\ge 0}W_k^2 \le \E \max_{n-1\ge k\ge 0}M_k^2 \le 96 n^2.
  \end{align}
\end{lemma}
\begin{proof}
    By Doob's inequality 
\cite[Theorem 10.9.4]{Gut}
for the reverse martingale $M_k$
and \eqref{tw9},
\begin{align}\label{lw1}
\E\max_{n_1\ge k\ge 0}M_k^2&
\le 4 \E M_0^2
\le 96n^2.
\end{align}
The result follows, recalling again $W_k\le M_k$.
\end{proof}

We show some further estimates used later.
\begin{lemma}
  \label{L2}
For $1\le k\le n-1$,
\begin{align}\label{l21}
  \P\xpar{Z_k\ge1} &\le \frac{2n}{k^2},
\\\label{l22}
  \P\xpar{Z_k\ge2}& \le \frac{24n^2}{k^4}.
\end{align}
\end{lemma}
\begin{proof}
By Markov's inequality,
  \eqref{zk}, and \eqref{tw3}, we have
  \begin{align}\label{l23}
\P\bigpar{Z_k\ge1\mid\cF_k} 
\le \E \bigpar{Z_k\mid\cF_k} 
%\le {Y_k}\frac{1}{k}
= \frac{Y_k}{k}\le \frac{W_k}{k^2}
  \end{align}
and
  \begin{align}\label{l24}
\P\bigpar{Z_k\ge2\mid\cF_k} 
\le \E \lrpar{\binom{Z_k}2\Bigm|\cF_k} 
= \binom{Y_k}2\frac{1}{k^2}
\le \frac{Y_k^2}{k^2}\le \frac{W_k^2}{k^4}.
  \end{align}
The results \eqref{l21}--\eqref{l22} 
follow by taking expectations, using \eqref{tm4} and \eqref{tw9}.
\end{proof}

\begin{lemma}  \label{LA}
For $1\le k\le n-1$,
\begin{gather}\label{la0}
A_{k-1}-A_k
\le \frac{W_k^2}{k^3},
\\\label{la}
\E A_k \le12\frac{n^2}{k^2}.
\end{gather}
\end{lemma}
\begin{proof}
By \eqref{tw4b} and Taylor's formula (or the Bonferroni inequalities),
  \begin{align}\label{la1}
A_{k-1}-A_k&
%= W_k-\E\bigpar{W_{k-1}\mid\cF_k}
=2k\Bigpar{\bigpar{1-\frac1k}^{Y_k}-1+\frac{Y_k}{k}}
%\notag\\&
\le 2k \binom{Y_k}2\frac{1}{k^2}
%\notag\\&
\le \frac{Y_k^2}{k}\le \frac{W_k^2}{k^3},
  \end{align}
which is \eqref{la0}.
As a consequence,
 \begin{align}\label{la2}
   A_k\le\sum_{i=k+1}^{n-1}\frac{W_i^2}{i^3}
 \end{align}
and thus, by \eqref{tw9},
 \begin{align}\label{la3}
\E A_k\le\sum_{i=k+1}^{n-1}\frac{\E W_i^2}{i^3}
\le24 n^2 \sum_{i=k+1}^{\infty}\frac{1}{i^3}
\le12\frac{n^2}{k^2}.
 \end{align}
\end{proof}

\section{Phase I: a Yule process}\label{SI}

In this section we consider the first part of the evolution of the red
dag $\DD_n$, and consider the variables $Y_{n-1},...,Y_{n_1}$,
%We consider first the (red) dag restricted to the set \set{n_1,\dots,n} of
%vertices, 
where (for definiteness) we let $n_1:=n_1\nn:=\floor{n/\log n}$.
(We might choose $n_1=n_1\nn$ as any (deterministic) sequence of integers
such that $n_1/n\to0$ slowly; in particular, any such sequence with $n_1\ge
n/\log n$ will also do.
We leave it to the reader to see precisely how small $n_1$ can be.)
%More precisely, we study the process $(Y_k)_k$ for $k\ge n_1$.
We will show that
the variables $Y_{n-1},...,Y_{n_1}$ can be approximated
(as \ntoo) by a time-changed \emph{Yule process}.

Recall that the {Yule process} is a continuous-time branching process,
where each particle lives a lifetime that has an exponential $\Exp(1)$
distribution, and then the particle splits into two new particles.
(All lifetimes are independent.) 
Let $\cY_t$ be the number of particles at time $t$.
The standard version, which we denote by $\cYY_t$, 
starts with one particle at time 0, but we 
start with $\cY_0=2$; thus the process 
$\cY_t$ can be seen as the sum of two independent
copies of the standard Yule process $\cYY_t$.

It is well known, see \eg{} \cite[Section III.5]{AN}, that for the standard
Yule process, the number of
particles at time $t$ has the geometric distribution $\Ge(\et)$ with
mean $e^t$ and
\begin{equation}
  \label{a1}
\P(|\cYY_t|=k) = \et\bigpar{1-\et}^{k-1},
\qquad k\ge1.
\end{equation}
Moreover,
$\cYY_t/e^t\asto\hxi$ as \ttoo, where (e.g.\ as a consequence of \eqref{a1})
$\hxi\in\Exp(1)$.
Hence, $\cY_t$ has a shifted negative binomial distribution
$\NegBin(2,\et)+2$ with 
\begin{equation}
  \label{a2}
\P(|\cY_t|=k) = (k-1)e^{-2t}\bigpar{1-\et}^{k-2},
\qquad k\ge2.
\end{equation}
In particular, for all $t\ge0$ we have
\begin{align}
  \label{th1}
\E\cY_t=
2\E\cYY_t
=2e^t, 
\end{align}
and, as \ttoo,
\begin{align}\label{thx}
e^{-t} \cY_t\asto\xi:= \hxi_1+\hxi_2\in\gG(2),
\end{align}
with $\hxi_1,\hxi_2\in\Exp(1)$ independent, so that their sum has 
a Gamma distribution.

We may also regard the Yule process $\cY$ as an infinite
tree (the \emph{Yule tree}), with
one vertex $\gam_0:=0$ (the root), and one vertex $\gam_i$ at each time a
particle 
splits (a.s.\ these times are distinct, and we may number them in increasing
order); each particle is then represented by an edge from its time of birth
to its time of death. Note that $\cY_t$, the number of living particles,
equals the number of edges alive at time $t$, and that the number of
particles that have died before (or at) $t$ is $\cY_t-2$.

We now change time by the mapping $t\mapsto\et$; thus the vertices in the
Yule tree are mapped to the points $e^{-\gam_i}\in\ooi$. The root is now at
1, and edges go from a larger label to a smaller.
If a particle is born at one of these times $x=e^{-\gam_i}$, and its lifetime
in the original Yule process is $\tau\in\Exp(1)$, then it lives there from
$\gam_i$ to $\gam_i+\tau$, and after the time change it is represented by an
edge from $x=e^{-\gam_i}$ to $e^{-(\gam_i+\tau)}=xe^{-\tau}=xU$, where
$U:=e^{-\tau}\in\Uoi$ has a uniform distribution.
Going backwards in time, we thus begin with two particles (edges) starting at 1.
Each edge starting at a point $x$ has endpoints $xU_x'$ and $xU_x''$, where
$U_x',U_x''\in\Uoi$, and all these uniform random variables are
independent. As before, we start two new edges at each endpoint.
We let $\hcY$ denote this (infinite) random tree with vertices in $\ooi$, and
let $\hcY_x$ be the number of particles (edges) alive at time $x$.

We may now compare the time-changed Yule tree to the red dag $\DD_n$
constructed above, scaled to \oi. 
An edge from a vertex $k$ 
ends at a vertex uniformly distributed on \set{1,\dots,k-1}, which we may
construct as $\floor{(k-1)U}+1$, where $U\in\Uoi$.
We thus start with one point at $n$, and add again two edges from it and
from the endpoint of every edge (except at 1), where now
an edge started at $j+1$ goes to $\floor{jU}+1$ with $U\in\Uoi$.
However, if two or more edges have the same endpoint, we still only start two
new edges there.

A point in $\DD_n$ that is $m$ generations away from the root, thus has
label
\begin{align}\label{tom1}
  X=\floor{\cdots\floor{(n-1)U_\nux1}\dotsm U_\nux{m}}+1
,\end{align}
for the some $U_\nux1,\dots,U_\nux{m}\in\oi$ (from the construction of the
edges), and then
\begin{align}\label{tom2}
  nU_\nux1\dotsm U_\nux{m}+1 \ge X \ge nU_\nux1\dotsm U_\nux{m} -m.
\end{align}
Let $\hhD_n$ denote the random red dag $\DD_n$
with all labels divided by $n$; thus the vertices are now points in $\ooi$.
We then see that $\hhD_n$ coincides with the time-changed Yule tree
up to small errors.
More precisely, we couple the two by first constructing the Yule tree $\cY$, and
its time-changed version $\hcY$, and then making a perturbation of $\hcY$ by
replacing each 
label $U_\nux{1}\dotsm U_\nux{m}$ by $X/n$ with $X$ as in \eqref{tom1}.
This gives a dag that coincides (in distribution) with $\hhD_n$ until the
first time that two edges in $\hhD_n$ have the same endpoint.

\begin{theorem}\label{TYule}
  We may \whp{} couple the random dag $\hhD_n$ and 
  the time-changed Yule tree $\hcY$, such that 
considering only vertices with
  labels in $[n_1/n,1]$, and edges with the starting point in this set,
there is a bijection between these sets of vertices in the two models 
which displaces each label by at most $\log^2 n/n$,
and a corresponding bijection between the edges (preserving the incidence
relations).
\end{theorem}

\begin{proof}
We have $\hcY_x = \cY_{-\log x}$ for every $x\in\ooi$, and thus by \eqref{th1}
\begin{align}\label{th3}
  \E\hcY_{x}=\E\cY_{-\log x} = 2e^{-\log x} = 2/x.
\end{align}
The number of vertices with labels in $[x,1]$ is $\hcY_x-1$, and taking
$x=n_1/n\sim 1/\log n$, we thus have $O_p(\log n)$ vertices; in particular
\whp{} less than $\log^2 n$ vertices. Consequently, \whp,
the number of generations from
the root to any point in $[n_1/n,1]$ is at most $\log^2n$, and then
the bound \eqref{tom2} shows that all vertex displacements are at most
$\llnn$.

Furthermore, it follows from \eqref{th3} that the expected number of
vertices in $\hcY$ that are within $\llnn$ from $n_1/n$ is 
\begin{align}\label{fb1}
 \E\bigpar{\hcY_{n_1/n-\llnn}-\hcY_{n_1/n+\llnn}}&
= \frac{2}{n_1/n-\llnn}-\frac{2}{n_1/n+\llnn}
\notag\\&
\sim\frac{4\llnn}{(n_1/n)^2}
=O\lrpar{\frac{\log^4n}{n}}=o(1),
\end{align}
and thus \whp{} no vertex is pushed across the boundary $n_1/n$ by the
displacements in the coupling.

Finally,
  it follows from \refL{L2} that the probability that two edges in the
  dag $\DD_n$ have the same   endpoint $k$ for some $k\ge n_1$ is at most
  \begin{align}\label{fb2}
    \sum_{k=n_1}^{n-1}\P(Z_k\ge2) 
\le 24n^2\sum_{k=n_1}^\infty k^{-4} =
    O\bigpar{n^2/n_1^3}
=o(1).
  \end{align}
Consequently, \whp{} the coupling above between $\hcY$ and $\DD_n$ yields a
bijection for vertices in $[n_1/n,n]$ and their edges.
\end{proof}

We define a random variable that will play an important role later: let
\begin{align}\label{Xi}
  \Xi=\Xi\nn:=\frac{W_{n_1}}n.
\end{align}

\begin{lemma}\label{LXi}
As \ntoo,
\begin{align}\label{lxi}
\Xi\nn=
\frac{W\nn_{n_1}}n\dto \xi\in\gG(2).
\end{align}
\end{lemma}
\begin{proof}
 We use the coupling in \refT{TYule} for each $n$, recalling \refR{Rcoupling}. 
Then, \whp,
  \begin{align}\label{ly1}
\hcY_{n_1/n+\llnn}\le Y_{n_1}\nn \le \hcY_{n_1/n-\llnn}
  \end{align}
and thus
\begin{align}\label{ly2}
  \bigabs{Y\nn_{n_1}-\hcY_{n_1/n}}\le {\hcY_{n_1/n-\llnn}-\hcY_{n_1/n+\llnn}}.
\end{align}
In particular, \eqref{fb1} implies
\begin{align}\label{ly3}
 \bigabs{Y\nn_{n_1}-\hcY_{n_1/n}}\pto0
.\end{align}
%and thus $Y\nn_{n_1}-\hcY_{n_1/n}\pto0$.
Moreover, 
\eqref{thx} implies 
\begin{align}\label{ly4}
  x\hcY_{x} =x\cY_{-\log x}
\asto \xi
\qquad \text{as $x\to0$},
\end{align}
with $\xi\in\gG(2)$.
Consequently, by \eqref{ly3} and \eqref{ly4},
\begin{align}\label{ly5}
\frac{n_1}n Y\nn_{n_1}
=
\frac{n_1}n \bigpar{Y\nn_{n_1}-\hcY_{n_1/n}} + \frac{n_1}n \hcY_{n_1/n}
\pto \xi.
\end{align}
Hence, recalling \eqref{tw3},
\begin{align}\label{ly6}
  \frac{W\nn_{n_1}}{n}
=\frac{n_1+1}{n_1}\cdot \frac{n_1}n Y\nn_{n_1}\pto\xi.
\end{align}
The convergence in probability in \eqref{ly5}--\eqref{ly6} depends on the
coupling used above, but it follows that convergence in distribution holds
also without it, which completes the proof.
\end{proof}

\section{Phase II: a boring flat part}\label{SII}
Let $n_2=n_2\nn$ be any sequence of integers with $\sqrt n \ll n_2 \le n_1$.
We will show that in the range $n_1\ge k\ge n_2$,
the variable $W_k$ essentially does not change, so it is equal to a random
constant. We begin with two lemmas valid for larger ranges. 

\begin{lemma}\label{LAII}
  As \ntoo,
\begin{align}\label{laii}
    \max_{n-1\ge k\ge n_2}\lrabs{\frac{A_k}{n}}
=\frac{A_{n_2}}{n}
\plito0.
  \end{align}
\end{lemma}
\begin{proof}
  By \refL{LA},
  \begin{align}
    \E \frac{A_{n_2}}{n}
\le 12 \frac{n}{n_2^2}=o(1),
  \end{align}
which shows \eqref{laii}, recalling \eqref{AAA}.
\end{proof}

\begin{lemma}\label{LM}
As \ntoo,
\begin{align}\label{lm2}
    \max_{n_1\ge k\ge 0}\lrabs{\frac{M_k}{n}-\Xi\nn}\plito0.
  \end{align}
\end{lemma}

\begin{proof}
    By Doob's inequality for the reverse martingale $M_k$
%\cite[Theorem 10.9.4]{Gut},
and using \eqref{tm2}, \eqref{tw7} and \eqref{tm4} as in \eqref{tw8}
(\cf{} the proof of \refL{LW*}),
\begin{align}\label{fb8}
\E\max_{n_1\ge k\ge 0}|M_k-M_{n_1}|^2&
\le 4 \E|M_0-M_{n_1}|^2
=4\sum_{i=1}^{n_1}\E\Var\bigpar{W_{i-1}\mid\cF_i}
\notag\\&
\le40\sum_{i=1}^{n_1}\E W_{i}
%\notag\\&
%\le 10(n-1-k)\cdot 2n
%\notag\\&
\le 80nn_1 = o(n^2).
\end{align}
We have, using
\eqref{Xi} and  $W_{n_1}=M_{n_1}-A_{n_1}$,
\begin{align}\label{lm3}
    \max_{n_1\ge k\ge 0}\lrabs{\frac{M_k}{n}-\Xi}
\le
    \max_{n_1\ge k\ge 0}\lrabs{\frac{M_k}{n}-\frac{M_{n_1}}n}
+\lrabs{\frac{A_{n_1}}n}
\plito0,
  \end{align}
where the convergence follows by \eqref{fb8} 
and \refL{LAII}.
\end{proof}

\begin{theorem}\label{TWXi}
As \ntoo,
  \begin{align}
    \max_{n_1\ge k\ge n_2}\lrabs{\frac{W_k}{n}-\Xi\nn}\plito0.
  \end{align}
\end{theorem}

\begin{proof}
We have, for any $k$,
\begin{align}\label{lm4}
\lrabs{\frac{W_k}{n}-\Xi}
\le
\lrabs{\frac{M_k}{n}-\Xi}
+\lrabs{\frac{A_{k}}n}
  \end{align}
and thus the result follows from \refLs{LAII} and \ref{LM}.
\end{proof}

\section{Phase III: deterministic decay from a random level}\label{SIII}

We extend the processes $W_k$, $M_k$ and $A_k$ to real arguments $t\in[0,n-1]$
by linear interpolation. Since the extended version $A_t$ is piecewise
linear, it is differentiable everywhere except at integer points, where
we (arbitrarily) take the left derivative.

\begin{lemma}\label{LdA}
  Let $\gd>0$.
Then
\begin{align}\label{lda}
  \E \lrsqpar{\sup_{\gd\sqrt n \le t \le n-1}\lrabs{\frac{\ddx}{\dd t}A_t}}
\le 
\frac{96}{\gd^3}n\qq
.\end{align}
\end{lemma}
\begin{proof}
Let $k:=\ceil{t}$, so  $k-1< t\le k$.
Then, by \eqref{la0},
\begin{align}\label{ky0}
0\le
 - \frac{\ddx}{\dd t} A_t =A_{k-1}-A_{k}
%= 2k\Bigpar{\Bigpar{1-\frac1k}^{Y_k}-1+\frac{Y_k}{k}}
\le \frac{W_k^2}{k^3}
\le \frac{W_k^2}{t^3}.
\end{align}
The result \eqref{lda} follows by \refL{LW*}.
\end{proof}

We rescale and define
\begin{align}\label{hA}
\hA\nn_t:=n\qw A_{t\sqrt n}\nn,
\qquad t\ge0.
\end{align}
Recall also that $C[a,b]$ is the (Banach) space of continuous functions on $[a,b]$.

\begin{lemma}\label{LA2}
  Let $0<\gd <b<\infty$.
Then the stochastic processes $\hA\nn_t$, $n\ge1$,
%\begin{align}
%\hA\nn_t:=n\qw A_{t\sqrt n}\nn
%\end{align}
are tight in $C[\gd,b]$.
\end{lemma}
\begin{proof}
We have, temporarily writing $A(t):=A\nn_t$,
\begin{align}\label{laa1}
 \ddt \hA\nn_t=n\qqw A'\xpar{t\sqrt n}.
\end{align}
Hence, \refL{LdA} yields
\begin{align}\label{lda2}
%  \E \sup_{\gd \le t \le b}\lrabs{\ddt n\qw A\nn_{t\sqrt n}}
  \E\lrsqpar{ \sup_{\gd \le t \le b}\lrabs{\frac{\ddx}{\dd t}\hA\nn_t}}
=\E \lrsqpar{n\qqw\sup_{\gd\sqrt n \le t \le b\sqrt n}\lrabs{\frac{\ddx}{\dd t}A_t}}
\le 
\frac{96}{\gd^3},
\end{align}
and thus the supremum in the \lhs{} forms a tight family of random variables
as $n$ varies.

Moreover, for a fixed $t\in[\gd,b]$ 
we have by \refL{LA}
\begin{align}\label{lax}
  \E\hA\nn_t =n\qw\E A\nn_{t\sqrt n}
\le n\qw\cdot 12\frac{n^2}{\floor{t\sqrt n}^2}
=12\xpar{1+o(1)}t^{-2}=O(1)
,
\end{align}and thus also the family $\hA\nn_t$ is tight.
The result follows, see
\cite[Theorem 8.2]{Billingsley}.
\end{proof}

\begin{theorem}\label{TIII}
We have
\begin{align}\label{tiii}
\sup_{0\le t \le b}\Bigabs{n\qw W_{t\sqrt n}\nn - t^2 \log\bigpar{1+\Xi\nn/t^2}}
\pto 0,
\end{align}
for every fixed $b>0$.
\end{theorem}
\begin{remark}
  We may note that \eqref{tiii} means convergence, in probability, in the
space $C\ooo$ with its standard topology (uniform convergence on compact sets).
Equivalently, we may consider the step functions 
$n\qw W_{\floor{t\sqrt n}}\nn$ and convergence in $D\ooo$.
\end{remark}

\begin{proof}
We divide the proof into several steps.

\stepx{A subsequence.}
  By \refL{LA2} and Prohorov's theorem \cite[Theorem 6.1]{Billingsley},
for every compact interval $[\gd,b]\subset\xooo$ we can find 
a subsequence $(n_\nu)$  such that, along the subsequence,
\begin{align}
  \label{kb1}
\hA\nn_t\dto \cA_{\gd,b}(t)
\qquad\text{in $C[\gd,b]$}
\end{align}
for some  continuous random function $\cA_{\gd,b}(t)$ on $[\gd,b]$.
Furthermore,
it suffices to consider a countable set of such intervals, for example 
$\cI:=\set{[m\qw,m]$, $m\ge2}$, and
by considering convergence in the product space
$\prod_{[\gd,b]\in\cI}C[\gd,b]$
we can find a subsequence such that \eqref{kb1} holds 
jointly for all compact intervals $[\gd,b]\in\cI$;
by adding a factor $\bbR^2$, we may also assume that this holds jointly with
\eqref{lxi} and \eqref{lm2}.
We consider until the last step of the proof only this subsequence.

\stepx{A coupling.}
By the Skorohod coupling theorem \cite[Theorem~4.30]{Kallenberg},
we may couple $D_n$ for different $n$ such that the convergence in
\eqref{kb1} holds \as{} for every $[\gd,b]\in\cI$,
and also \eqref{lxi} and \eqref{lm2}
hold a.s.
Since convergence in $C[\gd,b]$ means uniform convergence, this means that
\as{} $\hA\nn_t\to \cA_{\gd,b}(t)$ uniformly on $[\gd,b]$ for each
$[\gd,b]\in\cI$. 
It is evident that \as{}
the different limits $\cA_{\gd,b}(t)$ have to agree whenever intervals overlap,
and thus
there exists a continuous random function $\cA(t)$ defined on $\xooo$ such
that \as{} 
\begin{align}\label{kaa}
  \hA\nn_t\to\cA(t)
\end{align}
uniformly on  each compact interval $[\gd,b]\subset\xooo$.
(In other words, $\hA\nn_t\asto\cA(t)$ in the space $C\xooo$.)
Clearly, $\cA(t)\ge0$.
Furthermore, we now  \as{} have
\begin{gather}\label{kb2}
\Xi\nn\to\xi
\in\gG(2),
\\\label{kb3}
\sup_{0\le t\le n_1}\bigabs{n\qw M\nn_t -\Xi\nn}\to 0.
\end{gather}

It follows from \eqref{kaa}--\eqref{kb3} that \as,
\begin{align}\label{kb4}
  n\qw W_{t\sqrt n}\nn
=  n\qw M_{t\sqrt n}\nn - n\qw A_{t\sqrt n}\nn
=  n\qw M_{t\sqrt n}\nn - \hA\nn_t
\to \xi-\cA(t),
\end{align}
uniformly on each compact interval in $\xooo$.

\stepx{Identifying the limit.}
Since the limit in \eqref{kb4} is continuous, \eqref{kb4} and \eqref{tw3} yield,
again \as{} uniformly on each compact interval in $\xooo$,
\begin{align}\label{kb5}
  \frac{ Y_{\ceil{t\sqrt n}}\nn}{\ceil{t\sqrt n}}
%=   \frac{ W_{\ceil{t\sqrt n}}\nn}{\ceil{t\sqrt n}(\ceil{t\sqrt n}+1} 
=   \frac{ W_{\ceil{t\sqrt n}}\nn}{(t^2+o(1))n} 
\to
\cB(t):=
  t\qww\bigpar{ \xi-\cA(t)}.
\end{align}
By \eqref{laa1} and \eqref{tw4b},
with $k:=\ceil{t\sqrt n}$,
\begin{align}\label{kb6a}
  \ddt \hA\nn_t
=n\qqw\bigpar{A_k-A_{k-1}}
= -n\qqw\cdot 2k\Bigpar{\Bigpar{1-\frac1k}^{Y_k}-1+\frac{Y_k}{k}},
\end{align}
and thus \eqref{kb5} implies that a.s., uniformly on each compact interval in
$\xooo$,
\begin{align}\label{kb6}
  \ddt \hA\nn_t
\to - 2t\bigpar{e^{-\cB(t)}-1+\cB(t)}.
\end{align}
It follows from \eqref{kaa} and \eqref{kb6} that a.s., if $0<t_1<t_2<\infty$,
\begin{align}\label{kb7}
\cA(t_2)-\cA(t_1)=\lim_\ntoo\bigpar{\hA\nn_{t_2}-\hA\nn_{t_1}}
=\int_{t_1}^{t_2}\Bigpar{- 2t\bigpar{e^{-\cB(t)}-1+\cB(t)}}\dd t
.\end{align}
Consequently, a.s.\ the random function $\cA(t)$ is continuously
differentiable on $\xooo$, with derivative
\begin{align}\label{kb8}
\cA'(t)=- 2t\bigpar{e^{-\cB(t)}-1+\cB(t)},
\qquad 0<t<\infty
.\end{align}
This and the definition of $\cB(t)$
in \eqref{kb5} yield a differential equation for
$\cA(t)$, which we solve as follows.
First, let
\begin{align}\label{q1}
  \cC(t):=t^2\cB(t)=\xi-\cA(t).
\end{align}
Then
\begin{align}\label{q2}
 2t\cB(t)+t^2 \cB'(t) =\cC'(t)=-\cA'(t)=2t\bigpar{e^{-\cB(t)}-1+\cB(t)}
\end{align}
and thus
\begin{align}\label{q3}
 \cB'(t) =\frac{2}{t}\bigpar{e^{-\cB(t)}-1}
\end{align}
which can be written as
\begin{align}\label{q4}
  \frac{e^{\cB(t)}\dd\cB(t)}{e^{\cB(t)}-1}=-\frac{2\dd t}{t}
\end{align}
with the solution, for some $c\in\bbR$,
\begin{align}\label{q5}
  \log\bigpar{e^{\cB(t)}-1}= c-2\log t
\end{align}
and thus, with $C:=e^{c}>0$,
\begin{align}\label{q6}
  \cB(t)=\log\bigpar{1+C/t^2},
\qquad t>0.
\end{align}
Note that the constants $c$ and $C$ may be random.

We have shown that \eqref{q6} holds \as, for some random $C$, and
\eqref{kb5} then yields
\begin{align}\label{q7}
  \cA(t)=\xi-t^2\cB(t)
=\xi-t^2 \log\bigpar{1+C/t^2},
\qquad t>0.
\end{align}
It follows that a.s.
\begin{align}\label{q8}
  \cA(t)\to \xi-C
\qquad\text{as \ttoo}.
\end{align}
On the other hand, for every fixed $t>0$, as in \eqref{lax},
\begin{align}\label{laxa}
  \E\hA\nn_t %=n\qw\E A\nn_{t\sqrt n}
\le 12 \frac{n}{\floor{t\sqrt n}^2}
\to \frac{12}{t^2},
\end{align}
which by \eqref{kaa} and Fatou's lemma implies
\begin{align}\label{laxb}
  \E\cA(t) \le \frac{12}{t^2},
\qquad t>0.
\end{align}
In particular, $\cA(t)\pto0$ as \ttoo, which together with \eqref{q8}
yields $C=\xi$.

We thus have shown that \eqref{kaa}, \eqref{kb4} and \eqref{kb5} \as{} hold
uniformly on each compact interval in $\xooo$, with
\begin{align}\label{qa}
\cA(t)=\xi-t^2 \log\bigpar{1+\xi/t^2}.
\end{align}

\stepx{Convergence on $\ooo$.}
We extend the results just shown from $\xooo$ to $\ooo$ as follows.
First, note that $\cA(t)$ in \eqref{qa} extends to a continuous function on
$\ooo$, with $\cA(0)=\xi$.
We have $A\nn_0\le M\nn_0$ by \eqref{tm1},
and $n\qw M\nn_0\to \xi$ \as{}
by \eqref{kb2} and \eqref{kb3}.
Hence, a.s.,
\begin{align}\label{qb}
  \limsup_\ntoo n\qw A\nn_0
\le   \limsup_\ntoo n\qw M\nn_0
=
\xi = \cA(0).
\end{align}
On the other hand, 
for every $t>0$, 
$A\nn_{t\sqrt n}\le A\nn_0$ and thus
by \eqref{kaa}, a.s.
\begin{align}\label{qc}
  \liminf_\ntoo n\qw A\nn_0
\ge   \liminf_\ntoo n\qw A\nn_{t\sqrt n}
=\cA(t).
\end{align}
Letting $t\downto 0$ yields $\cA(t)\upto\cA(0)$ and thus
\begin{align}\label{qd}
  \liminf_\ntoo n\qw A\nn_0
\ge \cA(0).
\end{align}
Consequently, a.s., $\hA\nn_t\to\cA(t)$ for $t=0$ too.
We thus have \eqref{kaa} \as{} for each fixed $t\ge0$.
Since $\hA_t\nn$ and $\cA(t)$ are decreasing in $t$, and $\cA(t)$ is
continuous, this implies uniform convergence on each
compact interval $[0,b]\subset\ooo$.
It follows from \eqref{kb2}--\eqref{kb3} that \eqref{kb4} also holds \as{}
on each compact interval in $\ooo$, \ie, in $C\ooo$.
This means, by \eqref{qa},
\as{} uniformly on each compact interval,
\begin{align}\label{qe}
  n\qw W_{t\sqrt n}\nn
\to t^2 \log\bigpar{1+\xi/t^2}.
\end{align}
By \eqref{kb2} and the fact that $\frac{\ddx}{\dd x}\log(1+x)\le1$,
\eqref{qe} yields also
\begin{align}\label{qf}
  n\qw W_{t\sqrt n}\nn - t^2 \log\bigpar{1+\Xi\nn/t^2}
\to 0,
\end{align}
\as{} uniformly  on each compact interval in $\ooo$.

\stepx{Uncoupling.}
The \as{} convergence in \eqref{qf} depends on the chosen coupling of $D_n$ for
different $n$, but this yields \eqref{qf}
with convergence in probability in general, \ie, \eqref{tiii}.

\stepx{Conclusion.}
We have so far proved \eqref{tiii}
only for a subsequence, but the same proof shows that every subsequence has
a subsubsequence such that \eqref{tiii} holds,
which as is well known implies that \eqref{tiii} holds for the full sequence,
see \eg{} \cite[Section 5.7]{Gut}.
\end{proof}

\section{The number of descendants}\label{Spf1}

Recall that the random variable $X=X\nn$ is the number of 
descendants of $n$, \ie{} red vertices, and thus,
counting the root $n$ separately,
\begin{align}\label{r1}
  X=1+\sum_{k=1}^{n-1}J_k.
\end{align}
We make a Doob decomposition similar to \eqref{tm1}; 
in this case it takes the form,
since $J_k$ is $\cF_{k-1}$-measurable,
\begin{align}\label{r2}
  X=1+L_0+B_0,
\end{align}
where
\begin{align}\label{r3}
  L_k:=\sum_{i=k+1}^{n-1}\bigpar{J_i-\E(J_i\mid \cF_i)}
\end{align}
so that $(L_k)_0^{n-1}$ is a reverse martingale with $L_{n-1}=0$: 
$\E\xpar{L_{k-1}\mid \cF_k}=L_k$, 
and,
using \eqref{jk} and \eqref{zk},
\begin{align}\label{r4}
  B_k:=\sum_{i=k+1}^{n-1}\E(J_i\mid \cF_i)
=\sum_{i=k+1}^{n-1}\P(Z_i\ge1\mid \cF_i)
  =\sum_{i=k+1}^{n-1}\bigpar{1-(1-\tfrac1i)^{Y_i}}
\end{align}
is positive and increasing backwards:
\begin{align}\label{BBB}
0=B_{n-1}\le\dots\le B_1\le B_0.  
\end{align}

By \eqref{r4} and \refL{L2},
for every $k\le n-1$,
\begin{align}\label{r5}
%  \E \sum_{i=k+1}^{n-1}J_i = 
\E B_k 
=  \sum_{i=k+1}^{n-1}\E J_i 
=  \sum_{i=k+1}^{n-1}\P\bigpar{Z_i\ge1}
\le \sum_{i=k+1}^{\infty}\frac{2n}{i^2}
\le \frac{2n}{k}.
\end{align}
This is too coarse for small $k$; 
however, since $0\le J_i\le 1$ for every $i$,
we also have $B_0-B_\ell\le \ell$ for every $\ell\le n-1$.
Hence, \eqref{r5} implies
%using also \eqref{r2} and $\E L_k=0$,
\begin{align}\label{r6}
\E B_0 
\le \E B_{\ceil{\sqrt n}}+\ceil{\sqrt n}
\le 4\sqrt n.
\end{align}

Since $(J_i\mid \cF_i)$ is a Bernoulli variable,
$\Var\bigpar{J_{i}\mid\cF_i}\le \E\bigpar{J_{i}\mid\cF_i}$, and thus
the (reverse) martingale property of $(L_k)$ yields
\begin{align}\label{r7}
\E L_0^2 &
%=\Var L_0 
=\sum_{i=1}^{n-1}\E\bigsqpar{\Var\bigpar{J_{i}\mid\cF_i}}
\le \sum_{i=1}^{n-1}\E\bigsqpar{\E\bigpar{J_{i}\mid\cF_i}}
= \sum_{i=1}^{n-1}\E J_{i} 
=\E B_0
\le 4{\sqrt n}.
\end{align}
In particular, $L_0/\sqrt n \pto0$, which will show that $L_0$ is negligible
in \eqref{r2}.

\begin{lemma}\label{LX}
As \ntoo,
  \begin{align}\label{lx1}
\biggabs{\frac{X\nn}{\sqrt n} -\frac{\pi}2\sqrt{\Xi\nn}}
\pto 0.
  \end{align}
Thus,
\begin{align}\label{lx2}
  \frac{X\nn}{\sqrt n} \dto \frac{\pi}2\sqrt{\xi},
\end{align}
with $\xi\in\gG(2)$.
\end{lemma}

\begin{proof}
For convenience, we use the Skorohod coupling theorem as in the proof of
\refT{TIII}; 
we may thus assume that all \as{} convergence results in the proof of
\refT{TIII} hold.
(We may for simplicity  consider the same subsequence as in the proof of
\refT{TIII}, and then draw the conclusion for the full sequence as there;
alternatively, we may
argue that now when \refT{TIII} is proved, we may consider
the full sequence when we apply the Skorohod coupling theorem.)
In particular, \eqref{kb5} and \eqref{qa} (or \eqref{qe}) yield
\begin{align}\label{rb5}
  \frac{ Y_{\ceil{t\sqrt n}}\nn}{\ceil{t\sqrt n}}
%=   \frac{ W_{\ceil{t\sqrt n}}\nn}{\ceil{t\sqrt n}(\ceil{t\sqrt n}+1} 
=   \frac{ W_{\ceil{t\sqrt n}}\nn}{(t^2+o(1))n} 
\to
\cB(t)=
\log\bigpar{1+\xi/t^2}
\end{align}
\as{} uniformly on each compact interval in $\xooo$.

We extend $B_k$ to real arguments by linear interpolation and define also,
similarly to \eqref{hA} but with a different scaling, 
  \begin{align}\label{vs1}
    \hB\nn_t:=n\qqw B\nn_{t\sqrt n}.
  \end{align}
Then, with $k:=\ceil{t\sqrt n}$,
\begin{align}\label{vs2}
  \ddt\hB\nn_t = -\E \bigpar{J_k\mid\cF_k}
= \Bigpar{1-\frac{1}{k}}^{Y_k}-1
\end{align}
and thus it follows from \eqref{rb5} that, uniformly on each compact
interval in $\xooo$,
\begin{align}\label{vs3}
  \ddt\hB\nn_t \to e^{-\cB(t)}-1
= \frac{1}{1+\xi/t^2}-1
= -\frac{\xi}{\xi+t^2}.
\end{align}
Consequently, if $0<t_1<t_2<\infty$, \as{}
\begin{align}\label{vs4}
  \hB\nn_{t_1}-\hB\nn_{t_2}=-\int_{t_1}^{t_2}\ddt\hB\nn_t\dd t
\to \int_{t_1}^{t_2}\frac{\xi}{\xi+t^2}\dd t
= \sqrt{\xi}\Bigpar{\arctan\frac{t_2}{\sqrt\xi} -\arctan\frac{t_1}{\sqrt\xi}}
.\end{align}

Since %$J_k\in\set{0,1}$,  
\eqref{vs2} implies $\bigabs{\ddt\hB\nn_t}\le1$,
we have
\begin{align}\label{vs5}
  \bigabs{\hB\nn_0-\hB\nn_{t_1}}\le t_1.
\end{align}
Furthermore, \eqref{r5} implies, for $t_2\ge1$,
\begin{align}\label{vs6}
  \E\hB\nn_{t_2}=n\qqw\E B_{t_2\sqrt n} \le \frac{2\sqrt n}{\floor{t_2\sqrt n}}
\le\frac{4}{t_2}.
\end{align}
Thus, letting $t_1\to0$ and $t_2\to\infty$, we have
$\hB\nn_0-(\hB\nn_{t_1}-\hB\nn_{t_2})\pto0$, uniformly in $n$,
and it follows from \eqref{vs4} by standard arguments that
\begin{align}\label{vs7}
  \hB\nn_0\pto
\int_{0}^{\infty}\frac{\xi}{\xi+t^2}\dd t
=\frac{\pi}2\sqrt\xi.
\end{align}
Recall from \eqref{vs1} that $B_0=\sqrt n \hB\nn_0$.
The results \eqref{lx1}--\eqref{lx2} now follow from 
\eqref{vs7} by
\eqref{r2}, \eqref{r7}, and \eqref{kb2}.
\end{proof}

\begin{proof}[Proof of \refT{TX}, first part]
  The limit in distribution \eqref{tx1} follows immediately from \eqref{lx2},
using the well known facts that
$\chi_4^2\in\chi^2(4)$ and thus $\frac12\chi_4^2\in\gG(2)$,
see \refSS{SSnot},
and consequently
\begin{align}\label{semla}
  \sqrt\xi\eqd 2\qqw\chi_4.
\end{align}
\end{proof}

\section{Higher moments}\label{Smom}
In this section we prove some inequalities for higher moments.
We do not care about exact constants, and we use the convention that
$c_p$ stands for constants that may (and  will) depend on the
parameter $p$, but not on $n$; the value of $c_p$ may change from one
occurrence to another.

We consider first the reverse martingale $M_k$. We define
the maximal function
\begin{align}\label{mx}
  \Mx:=\max_{n-1\ge k\ge 0} M_k,
\end{align}
the martingale differences, for $n-1\ge k\ge 1$, 
recalling \eqref{tm2}, \eqref{tw3} and \eqref{tw1},
\begin{align}\label{gdm}
  \gD M_k&:= M_{k-1}-M_k 
=W_{k-1}-\E\bigpar{W_{k-1}\mid\cF_k}
\notag\\&\phantom:
=k \bigpar{Y_{k-1}-\E\bigpar{Y_{k-1}\mid\cF_k}}
\notag\\&\phantom:
=-k \bigpar{Z_k-\E\bigpar{Z_k\mid\cF_k}}
+2k \bigpar{J_k-\E\bigpar{J_k\mid\cF_k}}
,\end{align}
and the
conditional square function %$s(M)$ given by
\begin{align}\label{ssm}
 s(M):=
\lrpar{\sum_{i=1}^{n-1}\E\bigpar{\xpar{\gD M_i}^2\mid\cF_i}}\qq.
%\sum_{i=k}^{n-1}\E\bigpar{\xpar{M_{i-1}-M_i}^2\mid\cF_i}
%=\sum_{i=k}^{n-1}\E\bigpar{\bigpar{W_{i-1}-\E\xpar{W_{i-1}\mid\cF_i}}^2\mid\cF_i}
\end{align}
We use one of Burkholder's martingale inequalities 
\cite[Theorem 21.1]{Burkholder1973},
\cite[Corollary 10.9.1]{Gut}
on the martingale $M_k-M_{n-1}=M_k-2n$,
which yields
\begin{align}\label{burk}
  \E (\Mx)^p&
\le c_p(2n)^p + c_p\E \bigpar{\max_k|M_k-2n|}^p
\notag\\&
\le c_p n^p+c_p \E s(M)^p + c_p \E \bigpar{\max_k|\gD M_k|}^p
\notag\\&
\le c_pn^p+c_p \E s(M)^p + c_p \sum_{k=1}^{n-1}\E |\gD M_k|^p
.\end{align}
(This is valid for any $p>0$, although we only use $p\ge2$.)

\begin{lemma}\label{Lp}
  For every $p>0$,
\begin{align}\label{lp}
 \E \xpar{\Mx}^p  \le c_p n^p.  
\end{align}
\end{lemma}
\begin{proof}
  By Lyapunov's inequality, it suffices to prove \eqref{lp} for $p=2^j$,
 $j\ge1$   integer. 
We use induction on $j$. The base case $p=2$ is proved in \refL{LW*}.
In the rest of the proof, we thus assume $p\ge4$ and that \eqref{lp} holds
for the exponent $p/2$ (or smaller).
We use \eqref{burk}, and it remains to estimate the two last terms on its \rhs.

First, by \eqref{gdm} and \eqref{tw7},
\begin{align}\label{sw1}
\E\bigpar{\xpar{\gD M_i}^2\mid\cF_i}
=\Var\bigpar{W_{k-1}\mid\cF_k}
\le 10 W_k \le 10 M_k \le 10 \Mx.
\end{align}
Hence, \eqref{ssm} yields
\begin{align}\label{sw2}
  s(M)\le \sqrt{10 n \Mx}
\end{align}
and the induction hypothesis yields
\begin{align}\label{sw3}
 \E  s(M)^p\le c_p n^{p/2}\E\bigpar{\Mx}^{p/2}
\le c_p n^p.
\end{align}

Next, we recall the well known moment estimate for the binomial distribution
\begin{align}
  \E |\zeta-\E\zeta|^p \le c_p (Nq)^{p/2} + c_pNq,
\qquad \zeta\in\Bin(N,q).
\end{align}
(Coincidentally, this can be shown by the Burkholder inequality used in
\eqref{burk}, 
writing the binomial variable $\zeta-\E\zeta$ 
as a sum of $N$ independent centred Bernoulli variables.) 
Hence,
recalling the conditional distribution \eqref{zk},
we have 
\begin{align}\label{sw4}
  \E \bigpar{\bigabs{Z_k-\E(Z_k\mid\cF_k)}^p\mid\cF_k}&
\le c_p (Y_k/k)^{p/2} +c_p(Y_k/k)
\le c_p (W_k/k^2)^{p/2} +c_p(W_k/k^2)
\notag\\*&
\le c_p (\Mx/k^2)^{p/2} +c_p\Mx/k^2
\end{align}
and thus
\begin{align}\label{sw5}
  \E {\bigabs{Z_k-\E(Z_k\mid\cF_k)}^p}&
%\notag\\&
\le c_p k^{-p}\E(\Mx)^{p/2} +c_p k^{-2}\E\Mx
.\end{align}
Consequently, by the induction hypothesis, %and \eqref{lw*},
\begin{align}\label{sw6}
 k^p \E {\bigabs{Z_k-\E(Z_k\mid\cF_k)}^p}&
%\notag\\&
\le c_p n^{p/2} + c_pk^{p-2}n
\le  c_pn^{p-1}.
\end{align}
Similarly, since $J_k$ has a conditional Bernoulli distribution,
and using \eqref{l23},
\begin{align}\label{sw7}
  \E \bigpar{\bigabs{J_k-\E(J_k\mid\cF_k)}^p\mid\cF_k}&
\le c_p   \E \bigpar{\abs{J_k}^p\mid\cF_k}
= c_p  \E \bigpar{J_k\mid\cF_k}
\notag\\&
\le c_p W_k/k^2
\le c_p M_k/k^2
\le c_p \Mx/k^2
\end{align}
and thus, using again \eqref{lw*},
\begin{align}\label{sw8}
k^p\E \bigabs{J_k-\E(J_k\mid\cF_k)}^p&
\le c_p k^{p-2}  \E \Mx
\le c_p k^{p-2}n
\le c_p n^{p-1}.
\end{align}
Hence, \eqref{gdm}, %Minkowski's inequality and
\eqref{sw6} and \eqref{sw8} yield
\begin{align}\label{sw9}
\E|\gD M_k|^p&
\le c_p k^p\E \bigabs{Z_k-\E(Z_k\mid\cF_k)}^p
+ c_p k^p\E \bigabs{J_k-\E(J_k\mid\cF_k)}^p
\le c_p n^{p-1}.
\end{align}
The induction step is shown by \eqref{burk}, \eqref{sw3} and \eqref{sw9},
which completes the proof.
\end{proof}

We proceed to our main objective, the number $X$ of vertices in $\DD_n$.

\begin{lemma}\label{LXp}
  For every $p>0$,
  \begin{align}\label{lxp}
    \E (X\nn)^p \le c_p n^{p/2}.
  \end{align}
\end{lemma}

\begin{proof}
  We use the decomposition \eqref{r2}
and argue similary as in the proof of \refL{Lp}.
First, by \eqref{r4}, \cf{} \eqref{r5}--\eqref{r6},
\begin{align}\label{fy1}
  B_k \le \sum_{i=k+1}^{n-1}\frac{Y_i}{i}
\le \sum_{i=k+1}^{n-1}\frac{W_i}{i^2}
\le \Mx\sum_{i=k+1}^{\infty}\frac{1}{i^2}
\le \frac{\Mx}{k}
\end{align}
and thus
\begin{align}\label{fy2}
  B_0 \le B_{\ceil{\sqrt n}}+\ceil{\sqrt n}
\le n\qqw \Mx + 2n\qq.
\end{align}
%(Alternatively, we may optimize $k$ and get $B_0\le 3\sqrt{\Mx}$.)
Hence, by \refL{Lp},
for every $p>0$,
\begin{align}\label{fy3}
  \E B_0^p \le c_p n^{-p/2}\E(\Mx)^p + c_p n^{p/2}
\le c_p n^{p/2}.
\end{align}

Next, the conditional square function of the reverse martingale $L_k$ is
given by, see \eqref{r3} and \eqref{r4} and recall again that 
$(J_i\mid \cF_i)$ is a Bernoulli variable,
\begin{align}\label{fy4}
  s(L)^2 
= \sum_{i=1}^{n-1} \Var\bigpar{J_i\mid \cF_i}
\le \sum_{i=1}^{n-1} \E\bigpar{J_i\mid \cF_i}
=B_0
.\end{align}
Consequently, using \eqref{fy3}, for every $p>0$,
\begin{align}\label{fy5}
  \E s(L)^p \le \E B_0^{p/2} \le c_p n^{p/4}.
\end{align}
Furthermore, 
\begin{align}\label{fy6}
\gD L_k:=L_{k-1}-L_k = J_k-\E\bigpar{J_k\mid \cF_k}
\end{align}
and thus $|\gD L_k|\le 1$.
Consequently, the conditional Burkholder inequality in \eqref{burk} yields
\begin{align}\label{fy7}
  \E |L_0|^p
\le c_p \E s(L)^p + c_p \E\bigpar{\max_k|\gD L_k|}^p
\le c_p n^{p/4}.
\end{align}

The result \eqref{lxp} now follows from \eqref{r2}, \eqref{fy3} and \eqref{fy7}.
\end{proof}

\begin{proof}[Proof of \refT{TX}, conclusion]
\refL{LXp} shows that $\E|X\nn/\sqrt n|^p = O(1)$ for every fixed $p>0$.
By a standard argument, see{} \eg \cite[Theorems 5.4.2 and 5.5.9]{Gut},
this implies uniform integrability of the sequence
$|X\nn/\sqrt n|^p$  for every $p>0$ and thus convergence of all moments in
\eqref{lx2}. 
(Recall that convergence in distribution was proved in \refS{Spf1}.)

Finally,  \eqref{tx2}--\eqref{tx3} now follow from the formula
\begin{align}\label{chi2}
  \E \chi_4^r = 2^{r/2}\gG\Bigpar{\frac{r}2+2},
\end{align}
which is a simple  consequence of \eqref{chi1},
or of \eqref{semla} and \eqref{Gmom}.
This completes the proof.
\end{proof}

\section{Higher degree $d$}\label{Sd}

We have so far considered the random $2$-dag, with outdegree $d=2$.
The arguments and results above extend to any constant $d\ge2$ with minor
modifications which we sketch here, omitting straightforward details.
We let $d\ge2$ be fixed, and let $c$ and $c_p$ denote constants that may
depend on $d$ (and $p$); these may change value from one occurrence to the next.
Note that the case $d=2$ treated above is included as a special case below.

We define $Y_k$, $Z_k$, $J_k$, and $\cF_k$ as in \refS{Sbasic}; thus
$Y_{n-1}=d$,
\eqref{jk} and 
\eqref{zk} still hold, but \eqref{tw1} is replaced by
\begin{align}\label{tw1d}
  Y_{k-1}=Y_k-Z_k+dJ_k
%=Y_k-Z_k+2\cdot\indic{Z_k\ge1}.
.\end{align}
Then, instead of \eqref{tw2}--\eqref{tw2+},
\begin{align}\label{tw2d}
  \E \bigpar{Y_{k-1}\mid\cF_k}&
= Y_k -\E \bigpar{Z_k\mid \cF_k}+d\P\bigpar{Z_k\ge1\mid\cF_k}
\notag\\&
=Y_k-\tfrac{1}{k}Y_k+d\bigpar{1-(1-\tfrac1k)^{Y_k}}
\end{align}
and
\begin{align}\label{tw2+d}
  \E \bigpar{Y_{k-1}\mid\cF_k}&
\le Y_k-\frac{1}{k}Y_k + \frac{d}{k}Y_k 
=\frac{k+d-1}{k} Y_k.
\end{align}
We now define, letting $m\rise\ell:=m(m+1)\dotsm(m+\ell-1)$ denote the
rising factorial,
\begin{align}
  W_k:=(k+1)\rise{d-1}Y_k
=(k+1)\dotsm(k+d-1)Y_k
.\end{align}
Then, \eqref{tw2+d} yields
\begin{align}\label{tw4+d}
  \E \bigpar{W_{k-1}\mid\cF_k}&
\le W_k,
\end{align}
and thus again $W_k$ is a reverse supermartingale, with a Doob
decomposition \eqref{tm1} where now
\begin{align}
  M_{n-1}=W_{n-1}=dn\rise{d-1}
=O\bigpar{n^{d-1}}.
\end{align}
We still have \eqref{tw6}, up to the numerical constants 
(which depend on $d$), while  
we now have
\begin{align}\label{tw7d}
  \Var\bigpar{W_{k-1}\mid \cF_k}&
=O\bigpar{k^{2d-3}Y_k}
=O\bigpar{k^{d-2}W_k}
\end{align}
and
\begin{align}\label{tw9d}
  \E W_k^2 \le \E M_k^2 \le\E(\Mx)^2
= O\bigpar{n^{2d-2}}
.\end{align}

\refLs{L2}--\ref{LA} take the form
\begin{align}\label{l21d}
  \P\xpar{Z_k\ge1} & \le c \frac{n^{d-1}}{k^d},
\\\label{l22d}
  \P\xpar{Z_k\ge2}& \le c\frac{n^{2d-2}}{k^{2d}},
\\\label{la0d}
A_{k-1}-A_k &\le c\frac{W_k^{2}}{k^{d+1}},
\\\label{lad}
\E A_k &\le c\frac{n^{2d-2}}{k^d}.
\end{align}

The moment estimates in \refS{Smom} extend too.
We find $s(M)\le c\sqrt{n^{d-1}\Mx}$  and obtain by induction, for every
$p>0$,
\begin{align}\label{pia}
  \E (\Mx)^p \le c_p n^{p(d-1)}.
\end{align}
Instead of \eqref{fy1}--\eqref{fy3} and \eqref{fy7} we obtain
\begin{align}\label{pib}
  B_k &\le \Mx/k^{d-1},
\\\label{pic}
B_0&\le B_{\ceil{n^{(d-1)/d}}}+\ceil{n^{(d-1)/d}}
\le n^{-(d-1)^2/d}\Mx +2n^{(d-1)/d},
\\\label{pid}
\E B_0^p&\le c_p n^{p(d-1)/d},
\\\label{pie}
\E |L_0|^p&\le c_p n^{p(d-1)/(2d)},
\end{align}
and thus \eqref{r2} yields
\begin{align}\label{pif}
  \E \bigpar{X\nn}^p&\le c_p n^{p(d-1)/d}.
\end{align}

We may couple the initial phase of the dag with a branching process as in
\refS{SI}; however, now the each particle splits into $d$ new particles.
The corresponding standard process $\cYY_t$ starting with one particle has
\pgf, see \eg{} \cite[Remark III.5.1]{AN},
\begin{align}\label{pik}
  \E s^{\cYY_t}
= s e^{-t}\bigsqpar{1-\bigpar{1-e^{-(d-1)t}}s^{d-1}}^{-1/(d-1)},
\end{align}
which means that $(\cYY_t-1)/(d-1)$ has a negative binomial distribution 
$\NegBin\bigpar{\frac{1}{d-1},e^{-(d-1)t}}$. 
Since our version $\cY_t$ starts with $d$ particles as $t=0$,
and thus $\cY_t$ is the sum of $d$ independent copies of $\cYY_t$,
it follows that $(\cY_t-d)/(d-1)\in\NegBin\bigpar{\frac{d}{d-1},e^{-(d-1)t}}$. 
It also follows that
\begin{align}\label{pil}
\E\cY_t &= d \E\cYY_t 
=d+(d-1)\E \NegBin\Bigpar{\frac{d}{d-1},e^{-(d-1)t}}
\notag\\&
=d+(d-1)\frac{d}{d-1}\bigpar{ e^{(d-1)t}-1}
= de^{(d-1)t}
\end{align}
and thus, after the same time change as before,
\begin{align}\label{pim}
  \E\hcY_x =\E\cY_{-\log x}= d/ x^{d-1},
\qquad 0<x\le 1.
\end{align}
Moreover, also from \eqref{pik}, as \ttoo,
\begin{align}\label{pig}
  e^{-(d-1)t}\cYY_t\asto\gG\Bigpar{\frac{1}{d-1},d-1}
\end{align}
and thus, with $x=e^{-t}\to0$, 
\begin{align}\label{pih}
  e^{-(d-1)t}\cY_t=x^{d-1}\hcY_x\asto
\gG\Bigpar{\frac{d}{d-1},d-1}.
\end{align}

We may choose $n_1:=\floor{n/\log n}$ as in \refS{SI}, and then
\refT{TYule} holds, except that $\log^2n/n$ is replced by $\log^dn/n$.
Furthermore, we now have
\begin{align}\label{lxid}
\Xi\nn:=
\frac{W\nn_{n_1}}{n^{d-1}}\dto \xi\in\gG\Bigpar{\frac{d}{d-1},d-1}.
\end{align}

In \refS{SII}, we now choose $n_2\gg n^{(d-1)/d}$, and we have
\begin{align}\label{lm2d}
    \max_{n_1\ge k\ge 0}\lrabs{\frac{M_k}{n^{d-1}}-\Xi\nn}&\plito0,
%\end{align}
\\
%  \begin{align}
    \max_{n_1\ge k\ge n_2}\lrabs{\frac{W_k}{n^{d-1}}-\Xi\nn}&\plito0.
  \end{align}

In \refS{SIII}, we define
\begin{align}\label{hAd}
\hA\nn_t:=n^{-(d-1)} A_{t n^{(d-1)/d}}\nn,
\qquad t\ge0.
\end{align}
Then tightness holds as in \refL{LA2},  and we can argue as in the proof of
\refT{TIII} using a suitable subsequence and a suitable coupling.
Then \eqref{kaa} holds \as{} uniformly on compact intervals, 
and \eqref{kb5} becomes
\begin{align}\label{kb5d}
  \frac{ Y_{\ceil{tn\xdd}}\nn}{\ceil{tn\xdd}}
=   \frac{ W_{\ceil{t n\xdd}}\nn}{(t^d+o(1))n^{d-1}} 
\to
\cB(t):=
  t^{-d}\bigpar{ \xi-\cA(t)}.
\end{align}
This leads by the arguments above to 
\begin{align}\label{kb6d}
  \ddt \hA\nn_t
\to - dt^{d-1}\bigpar{e^{-\cB(t)}-1+\cB(t)},
\end{align}
uniformly on compact intervals in $(0,\infty)$,
and then to
the differential equation
(instead of \eqref{kb8})
\begin{align}\label{kb8d}
\cA'(t)=- dt^{d-1}\bigpar{e^{-\cB(t)}-1+\cB(t)},
\qquad 0<t<\infty
,\end{align}
with the solution
%\eqref{q6}
\begin{align}\label{q6d}
  \cB(t)=\log\bigpar{1+C/t^d},
\qquad t>0,
\end{align}
where
again we find $C=\xi$ \as, and consequently
\begin{align}\label{qad}
\cA(t)=\xi-t^d \log\bigpar{1+\xi/t^d}.
\end{align}
Finally, we extend the convergence to $\ooo$ as above, 
and reach the conclusion that (generalizing \refT{TIII})
\begin{align}\label{tiiid}
\sup_{0\le t \le b}\Bigabs{n^{-d} W_{t n\xdd}\nn - t^d \log\bigpar{1+\Xi\nn/t^d}}
\pto 0,
\end{align}
for every fixed $b>0$.

In \refS{Spf1}, 
we  replace \eqref{rb5} by
\begin{align}\label{rb5d}
  \frac{ Y_{\ceil{tn\xdd}}\nn}{\ceil{tn\xdd}}
=   \frac{ W_{\ceil{t n\xdd}}\nn}{(t^d+o(1))n^{d-1}} 
\to
\cB(t)=
\log\bigpar{1+\xi/t^d}
\end{align}
and define
\begin{align}
  \label{vs1d}
\hB\nn_t:=  n\xddw B\nn_{tn\xdd}
.\end{align}
This leads to
\begin{align}\label{vs7d}
  n\xddw B\nn_{0}\pto\intoo\frac{\xi}{\xi+t^d}\dd t=
  \xi^{1/d}\intoo\frac{1}{1+x^d}\dd x
=\frac{\pi}{d\sin(\pi/d)}\xi^{1/d},
\end{align}
where the integral is evaluated by a substitution yielding a 
Beta integral \cite[5.12.3, together with 5.12.1 and 5.5.3]{NIST}:
\begin{align}
 \intoo\frac{1}{1+x^d}\dd x&
= \frac{1}{d}\intoo\frac{y^{\frac{1}{d}-1}}{1+y}\dd y
=\frac{1}{d}B\Bigpar{\frac{1}{d},\frac{d-1}{d}}
=\frac{1}{d}\gG\Bigpar{\frac{1}{d}}\gG\Bigpar{\frac{d-1}{d}}
\notag\\&
=\frac{\pi}{d\sin(\pi/d)}
.\end{align}

\begin{proof}[Proof of \refT{TD}]
  The limit in distribution \eqref{td1} follows from 
\eqref{r2}, \eqref{pie} and \eqref{vs7d},
recalling  \eqref{lxid} and writing $\xi=(d-1)\gamma$.
Moment convergence then follows from the bounds \eqref{pid} and \eqref{pie}
as in the case $d=2$. Finally, the moment convergence and \eqref{Gmom} yield
\eqref{tdmom}.
\end{proof}

\section{Further results}\label{Sfurther}
We give here some further results on the structure of the random dag $\DD_n$.
Again, we consider for simplicity only the case $d=2$, and leave the
straightforward extensions to larger $d$ to the reader.

\subsection{Density of descendants}\label{SSdensity}

The proof of \refT{TX} shows that most vertices in $\DD_n$
are in the range $O\bigpar{\sqrt n}$.
More preciesely, 
let $0\le a\le b\le\infty$, and let 
\begin{align}\label{fa1}
X_{a,b}\nn:=\bigabs{\DD_n\cap(a\nqq,b\nqq]},  
\end{align}
the number of 
descendants of $n$ (red vertices) in the interval $(a\nqq,b\nqq]$.
(Thus, $X\nn=X_{0,\infty}\nn$.)
Then, \refL{LX} can be extended:
\begin{lemma}  \label{LXX}
If\/ $0\le a\le b\le\infty$ are fixed, then as \ntoo,
\begin{align}\label{fa2}
 {\frac{X_{a,b}\nn}{\sqrt n} - \int_a^b \frac{\Xi\nn}{\Xi\nn+t^2}\dd t}
\lito0
\end{align}
and thus
\begin{align}\label{fa3}
 \E\biggpar{\frac{X_{a,b}\nn}{\sqrt n}\biggm|\Xi\nn} 
- \int_a^b \frac{\Xi\nn}{\Xi\nn+t^2}\dd t
\lito0
\end{align}
and, unconditionally,
\begin{align}\label{fa4}
\frac{\E X_{a,b}\nn}{\sqrt n} \to \int_a^b p(t)\dd t,
\end{align}
where
\begin{align}\label{fa5}
  p(t):=
\E\frac{\xi}{\xi+t^2}
=\intoo \frac{x^2}{x+t^2}e^{-x}\dd x.
\end{align}
\end{lemma}
\begin{proof}
  If $0<a\le b<\infty$, 
let $k_a:=\floor{a\nqq}$ and $k_b:=\floor{b\nqq}$.
Then \eqref{r3}--\eqref{r4} show that, 
provided $n$ is so large that $b\nqq<n$, 
\begin{align}\label{fa6}
  X\nn_{a,b}
=\sum_{k=k_a+1}^{k_b} J_k
=B_{k_a}-B_{k_b}+L_{k_a}-L_{k_b}.
\end{align}
Convergence in probability in \eqref{fa2} then follows
from \eqref{vs1} and \eqref{vs4} together with \eqref{r7} 
(and, for example, Doob's inequality),
and as always \eqref{kb2}.
If $a=0$ or $b=\infty$, this result follows similarly using also
\eqref{vs5}--\eqref{vs6} as in the proof of \refL{LX}. 

Thus, \eqref{fa2} holds in probability. 
This implies convergence also in $L^1$, since
uniform integrability holds because 
$X\nn_{a,b}/\sqrt n\le X\nn/\sqrt n$ and, recalling \eqref{Xi}, 
\begin{align}\label{fa7}
\int_a^b \frac{\Xi\nn}{\Xi\nn+t^2}\dd t
\le \intoo \frac{\Xi\nn}{\Xi\nn+t^2}\dd t=\frac{\pi}2\sqrt{\Xi\nn}
\le \frac{\pi}2\sqrt{\frac{\Mx}{n}}
\end{align}
and these are uniformly integrable by \refLs{LXp} and \refL{Lp}.

Next, \eqref{fa3} follows from \eqref{fa2} by taking the conditional
expectation,
and \eqref{fa4} follows  by taking the unconditional expectation,
using \eqref{lxi} and Fubini's theorem,
and again the uniform integrability of \eqref{fa7}.
The final equality in \eqref{fa4} follows since $\xi\in\gG(2)$ has density
function $xe^{-x}$ by \eqref{Gamma}.
\end{proof}

\begin{remark}
The function $p(t)$ can be expressed using the exponential integral
$E_1(x)$, % and $\Ei(x)$,
see \cite[6.2.1-2 and 6.7.1]{NIST}:
\begin{align}
    p(t)=\intoo\Bigpar{x-t^2+\frac{t^4}{x+t^2}}e^{-x}\dd x
=1-t^2+t^4e^{t^2} E_1(t^2)
%=1-t^2-t^4e^{t^2}\Ei(-t^2)
.\end{align}
\end{remark}

Informally, \refL{LXX} says that, asymptotically, the density of descendants
of $n$ around any $k < n$ is $\Xi/(\Xi+k^2/n)$ conditioned on $\Xi$, and
$p(k/\sqrt n)$ unconditionally.
Another aspect of this is the following theorem, where we consider a single
vertex $k$.

\begin{theorem}\label{Tden}
  Conditioned on $\Xi\nn$, the probability that vertex $k$ 
is a descendant of $n$ (\ie, belongs to $\DD_n$)
is
\begin{align}\label{tden1}
  \P\bigpar{J\nn_k=1\mid\Xi\nn}
=
\frac{\Xi\nn}{\Xi\nn+k^2/n}+\op(1),
\end{align}
uniformly in $k\le n_1$.
Hence, the unconditional probability is,
with $p(t)$ given by \eqref{fa5},
\begin{align}\label{tden2}
  \P\bigpar{J\nn_k=1}
=
\E\frac{\xi}{\xi+k^2/n}+o(1)
=
p\bigpar{k/\sqrt n}
+o(1),
\end{align}
uniformly in $k\le n_1$.
\end{theorem}

\begin{proof}
Recall that $J_k=\indic{k\in\DD_n}$ is a Bernoulli variable; hence 
$\P(J_k=1)=\E J_k$, and
this holds also conditionally.

  Consider first $k\in [\gd \nqq,b \nqq]$ for some fixed $0<\gd<b<\infty$,
and let $t:=k/\sqrt n\in[\gd,b]$.
Using again the simplifying assumptions in the proof of \refT{TIII}, 
we see from \eqref{vs2}--\eqref{vs3} and \eqref{kb2} that
\begin{align}\label{sel3}
 \max_{k\in[\gd \nqq,b\nqq]}
\Bigabs{\E\xpar{J_k\mid \cF_k} - \frac{\Xi\nn}{\Xi\nn+k^2/n}}
\pto 0.
\end{align}
Hence, by dominated convergence (the $\max$ is bounded by 1),
\begin{align}\label{sel4}
\E \max_{k\in[\gd \nqq,b\nqq]}
\Bigabs{\E\xpar{J_k\mid \cF_k} - \frac{\Xi\nn}{\Xi\nn+k^2/n}}
\to 0.
\end{align}
Thus, by taking the conditional expectation with respect to $\Xi\nn$,
assuming that $n$ is so large that $b\sqrt n\le n_1$ and thus
$\E(J_k\mid\Xi\nn)=\E\bigsqpar{\E(J_k\mid\cF_k)\mid\Xi\nn}$,
\begin{align}\label{den5}
&\max_{k\in[\gd \nqq,b\nqq]}
\Bigabs{\E\xpar{J_k\mid \Xi\nn} - \frac{\Xi\nn}{\Xi\nn+k^2/n}}
\notag\\&\hskip4em
\le\max_{k\in[\gd \nqq,b\nqq]}\E
\Bigpar{\Bigabs{\E\xpar{J_k\mid \cF_k} -
  \frac{\Xi\nn}{\Xi\nn+k^2/n}}\,\Bigm|\Xi\nn}
\notag\\&\hskip4em
\le
\E \Bigpar{\max_{k\in[\gd \nqq,b\nqq]}
\Bigabs{\E\xpar{J_k\mid \cF_k} - \frac{\Xi\nn}{\Xi\nn+k^2/n}}\,\Bigm|\Xi\nn}
\lito 0,
\end{align}

Furthermore, if $n_1\ge k> \ell\ge1$, then when the evolution comes to $k$, we
have $Y_k$ red edges, and each of them ends at $k$ with probability $1/k$.
We have the same probability for each of these edges to end at $\ell$
instead, and
since the endpoints are independent, we see that conditioned on $\cF_k$,
$Z_\ell$ is stochastically larger than $Z_k$. (Larger, since there may also be
red edges ending at $\ell$ that start at $k$ or later.)
Hence, 
\begin{align}
  \E(J_\ell\mid\cF_k)=
  \P(Z_\ell\ge1\mid\cF_k)
\ge 
  \P(Z_k\ge1\mid\cF_k)
=\E(J_k\mid\cF_k)
\end{align}
and thus
$\E(J_\ell\mid\Xi\nn)\ge \E(J_k\mid\Xi\nn)$.
In other words,
$\E(J_k\mid\Xi\nn)$ is decreasing in $k\in[1,n_1]$. The same obviously
holds for
$\Xi\nn/(\Xi\nn+k^2/n)$.
Consequently, with $k_b:=\floor{b\sqrt n}$,
\begin{align}\label{den6}
\max_{k\in[b\nqq,n_1]}&
\Bigabs{\E\xpar{J_k\mid \Xi\nn} - \frac{\Xi\nn}{\Xi\nn+k^2/n}}
\notag\\
&\le\max_{k\in[b\nqq,n_1]}
\Bigpar{\E\xpar{J_k\mid \Xi\nn} + \frac{\Xi\nn}{\Xi\nn+k^2/n}}
\notag\\
&\le\Bigabs{\E\xpar{J_{k_b}\mid \Xi\nn} - \frac{\Xi\nn}{\Xi\nn+k_b^2/n}}
+2\frac{\Xi\nn}{\Xi\nn+k_b^2/n}.
\end{align}
Hence, using \eqref{den5},
\begin{align}\label{den7}
\limsup_\ntoo&
\E\max_{k\in[b\nqq,n_1]}
\Bigabs{\E\xpar{J_k\mid \Xi\nn} - \frac{\Xi\nn}{\Xi\nn+k^2/n}}
\notag\\&
\le
2\limsup_\ntoo\E\frac{\Xi\nn}{\Xi\nn+k_b^2/n}
=2\E \frac{\xi}{\xi+b^2},
\end{align}
which can be made arbitrarily small by choosing $b$ large
(by dominated convergence).
Similarly, 
with $k_\gd:=\ceil{\gd\sqrt n}$,
\begin{align}\label{den8}
\max_{k\in[1,\gd \nqq]}&
\Bigabs{\E\xpar{J_k\mid \Xi\nn} - \frac{\Xi\nn}{\Xi\nn+k^2/n}}
\notag\\
&\le\max_{k\in[1,\gd \nqq]}
\biggpar{\Bigpar{1-\E\xpar{J_k\mid \Xi\nn}} 
 + \Bigpar{1-\frac{\Xi\nn}{\Xi\nn+k^2/n}}}
\notag\\
&\le\Bigabs{\E\xpar{J_{k_\gd}\mid \Xi\nn} - \frac{\Xi\nn}{\Xi\nn+k_\gd^2/n}}
+2\Bigpar{1-\frac{\Xi\nn}{\Xi\nn+k_b^2/n}}
\end{align}
and thus
\begin{align}\label{den9}
\limsup_\ntoo&
\E\max_{k\in[1,\gd \nqq]}
\Bigabs{\E\xpar{J_k\mid \Xi\nn} - \frac{\Xi\nn}{\Xi\nn+k^2/n}}
\notag\\&
\le
2\limsup_\ntoo\E\frac{k_\gd^2/n}{\Xi\nn+k_\gd^2/n}
=2\E \frac{\gd^2}{\xi+\gd^2},
\end{align}
which can be made arbitrarily small by choosing $\gd$ small.

It follows from \eqref{den5}, \eqref{den7} and \eqref{den9} that
\begin{align}\label{den10}
\max_{k\in[1,n_1]}&
\Bigabs{\E\xpar{J_k\mid \Xi\nn} - \frac{\Xi\nn}{\Xi\nn+k^2/n}}
\lito 0,
\end{align}
which is a more precise version of \eqref{tden1}

Finally,  \eqref{kb2} implies
\begin{align}
\sup_{t>0} \lrabs{\frac{\Xi\nn}{\Xi\nn+t^2}-\frac{\xi}{\xi+t^2}}
=
\sup_{t>0} {\frac{t^2\abs{\Xi\nn-\xi}}{(\Xi\nn+t^2)(\xi+t^2)}}
\le
\frac{|\Xi\nn-\xi|}{\xi}
\asto0,
\end{align}
and thus, by dominated convergence,
\begin{align}
\sup_{t>0} \lrabs{\E\frac{\Xi\nn}{\Xi\nn+t^2}-\E\frac{\xi}{\xi+t^2}}
\le
\E\sup_{t>0} \lrabs{\frac{\Xi\nn}{\Xi\nn+t^2}-\frac{\xi}{\xi+t^2}}
\to0,
\end{align}
Hence, 
taking the expectation in \eqref{den10} yields \eqref{tden2}.
\end{proof}

\subsection{Different $n$ yield asymptotically independent
  results}\label{Scoupling} 
As noted in \refR{Rcoupling}, the construction naturally constructs 
the dags $D_n$
for all $n$ together. Using this coupling, we may consider the joint
distribution of, for example, $X\nn$ and $X\nni$.
Somewhat surprisingly, $X\nn$ and $X\nni$ are asymptotically independent:

\begin{theorem}\label{Tnni}
  As \ntoo,
  \begin{align}
\bigpar{X\nn/\sqrt{n},X\nni/\sqrt{n+1}}\dto \bigpar{\zeta,\zeta'},
  \end{align}
where $\zeta$ and $\zeta'$ are independent copies of the limit
$(\pi/\sqrt8)\chi_4$ in \eqref{tx1}.
\end{theorem}

\begin{proof}
  Consider the evolutions of the red dags $\DD_n$ and $\DD_{n+1}$ together,
starting at $n$ and $n+1$ and going down, as always; 
these evolutions are independent until they first have a common vertex.
The probability that $k$ is the first common vertex is thus at most the
probability that two independent versions of $\DD_n$ and $\DD_{n+1}$ both
contain $k$, which by \eqref{l21} is
\begin{align}
  \begin{cases}
    \P\bigpar{Z\nn_k\ge1}\P\bigpar{Z\nni_k\ge1}\le \frac{4n(n+1)}{k^4}
\le c\frac{n^2}{k^4}, & k<n,
\\
    \P\bigpar{Z\nni_n\ge1}\le \frac{2(n+1)}{n^2},& k=n.
  \end{cases}
\end{align}
Consequently, the probability that $\DD_n$ and $\DD_{n+1}$ meet before
$n\nn_1$
is
\begin{align}
  \le \sum_{k=n_1}^{n-1} c\frac{n^2}{k^4}+\frac{c}{n}
\le c \frac{n^2}{n_1^3} +\frac{c}{n}= o(1).
\end{align}
Consequently, \whp{} $\DD_n$ and $\DD_{n+1}$ are independent until $n_1$;
more formally, we may couple the pair $\bigpar{\DD_n,\DD_{n+1}}$ with a pair
$\bigpar{\DD'_n,\DD'_{n+1}}$ of independent copies of them such that the two
pairs \whp{} coincide until $n_1\nn$. In particular, this and the definition
\eqref{Xi} show that the pair $(\Xi\nn,\Xi\nni)$ can be coupled 
with a pair of independent copies of them
(defined in the same way from $\DD'_n$ and $\DD'_{n+1}$)
such that the two pairs coincide \whp 
Consequently, \refL{LXi} implies that
\begin{align}
  \bigpar{\Xi\nn,\Xi\nni}\dto\bigpar{\xi,\xi'},
\end{align}
where $\xi,\xi'\in\gG(2)$ are independent.
The result then follows by \eqref{lx1}.
\end{proof}

This result may seem surprising, since we have seen that most vertices $k$ in
$\DD_n$ and $\DD_{n+1}$ have $k$ of the order $\sqrt n$, and that in this
range, the density of vertices is high, which means that $\DD_n$ and
$\DD_{n+1}$ necessarily have a large number of common vertices.
Since the $\DD_n$ and $\DD_{n+1}$ have the same descendants of any common vertex,
it follows that the graphs $\DD_n$ and $\DD_{n+1}$ are strongly dependent.
Nevertheless, the proof above shows that $\DD_n$ and $\DD_{n+1}$ are
essentially independent in the first phase, which determines $\Xi\nn$ and
$\Xi\nni$.
Almost all vertices that contribute to $X\nn$ and $X\nni$ are
in the later dense phase, where there is  strong dependence,
but this does not prevent the asymptotic independence of $X\nn$ and $X\nni$
because in this phase, there are so many vertices and edges that the
evolution is governed by a law of large numbers and is essentially
deterministic; 
hence the strong dependence here does not matter.

\begin{remark}\label{Rnni}
  We considered above $X\nn$ and $X\nni$ only to be concrete.
The result extends to $X^{(n'_\nu)}$ and $X^{(n''_\nu)}$ for any two
sequences $n'_\nu$ and $n''_\nu$ that tend to infinity, with $n'_\nu< n''_\nu$.
(This follows by the same proof, where we treat the cases 
$(n''_\nu)_1 \le n'_\nu$ and $(n''_\nu)_1 > n'_\nu$ separately in the first part.)

Furthermore, the theorem extends to any finite number of such sequences.
\end{remark}

\begin{remark}\label{Rnoas}
  \refT{Tnni} shows that the sequence $X\nn/\sqrt n$ does \emph{not}
  converge \as; in contrast, the theorem (with \refR{Rnni})
implies that this sequence \as{}   oscillates wildly.
\end{remark}

It seems interesting to understand the relations between $\DD\nn$ and
$\DD\nni$ further. For example,
consider the number of common vertices 
\begin{align}
\gU\nn:=\bigabs{\DD\nn\cap\DD\nni}.
\end{align}
We have
$\Upsilon\nn\le\min\bigpar{X\nn,X\nni}$, and thus \refT{Tnni} implies (since
uniform integrability of $\min\bigpar{X\nn,X\nni}/\sqrt n$ follows from 
\refL{LXp}, or indeed from \refT{TX})
\begin{align}
\limsup_\ntoo
  \E\bigpar{ \Upsilon\nn/\sqrt n} \le \E\bigsqpar{\min\bigpar{\zeta,\zeta'}}
=\intoo\P(\zeta>x)^2\dd x
=\frac{27\pi^{3/2}}{64\sqrt2}.
\end{align}
(We omit the final, straightforward calculation.)
On the other hand, it follows easily from the results in \refSS{SSdensity} that
\begin{align}
\liminf_\ntoo
  \E\bigpar{ \Upsilon\nn/\sqrt n} >0.
\end{align}

\begin{problem}
 What is the asymptotics of the number of common vertices 
$\gU\nn:=\bigabs{\DD\nn\cap\DD\nni}$?
(I.e., the vertices that are descendants of both $n$ and $n+1$.)
\\
We conjecture that $  \E\bigpar{\Upsilon\nn/\sqrt n} \to \upsilon$ for some
constant $\gu>0$. Show this! What is $\gu$?
What is the asymptotic distribution of $\Upsilon\nn/\sqrt n$? (Assuming that
it exists.)
\end{problem}

\section{Some variations}\label{Svar}
We consider here the two variations mentioned in the introduction, and show
that the same results hold for them too.

\subsection{Several roots}\label{SSroots}
We may start with any given number $m\ge1$ roots, and then add $n-m$
vertices with outdegree $d$ recursively as above.
(We assume $1\le m\le n$.)
Denote the resulting random $d$-dag by $D_{n,m}$, and let $\DD_{n,m}$ be the
subgraph consisting of all vertices and edges that can be reached from $n$. 

Note that $D_{n,m}$ can be obtained from $D_n$ by simply removing all edges
between the roots, i.e., all edges within $[1,m]$.
Consequently, 
$D_{n,m}$ and $D_n$ have the same descendants in the interval $(m,n]$, 
and it follows that
\begin{align}\label{ask1}
  |\DD_n| -m < |\DD_{n,m}|\le |\DD_n|.
\end{align}

\begin{theorem}
  If the process starts with $m=o\bigpar{n\xdd}$ roots, 
and we thus define $X\nn:=|\DD_{n,m}|$,
then the results in
  \refTs{TX} and \ref{TD} still hold.
\end{theorem}
\begin{proof}
  An immediate consequence of \eqref{ask1}.
\end{proof}

We may also obtain results for larger $m$. For simplicity we consider only
the case $d=2$.
Define, for $\mu>0$ and $x>0$,
\begin{align}\label{ask}
  \psi_\mu(x)&:
=\intoo\frac{x}{x+\max\set{t,\mu}^2}\dd t
=\int_\mu^\infty\frac{x}{x+t^2}\dd t+\frac{\mu x}{x+\mu^2}
\notag\\*&\phantom:
=
\sqrt{x}\arctan\frac{\sqrt x}{\mu}
+\frac{\mu x}{x+\mu^2}
.\end{align}

\begin{theorem}\label{Tm2}
Let $d=2$.
 Suppose that $m=m_n\to\infty$ such that $m/\sqrt n\to \mu\in(0,\infty)$.
Then
\begin{align}\label{tm2a}
  \frac{\abs{\DD_{n,m}}}{\sqrt n} \dto \psi_\mu\bigpar{\xi},
\end{align}
with convergence of all moments, where $\xi\in\gG(2)$.
Moreover,
\begin{align}\label{tm2b}
  \frac{\abs{\DD_{n,m}}}{\sqrt n} - \psi_\mu\bigpar{\Xi\nn}
\pto0.
\end{align}
\end{theorem}

\begin{proof}
  The number of non-roots in $\DD_{n,m}$ is, using the notation \eqref{fa1},
\begin{align}\label{ask2}
\bigabs{\DD_{n,m}\cap(m,n]}    
=
\bigabs{\DD_{n}\cap(m,n]} 
=X\nn_{m/\sqrt n,\infty},
\end{align}
and thus it follows from \refL{LXX} that
\begin{align}\label{ask3}
\frac{\abs{\DD_{n,m}\cap(m,n]}}{\sqrt n} - 
\int_\mu^\infty \frac{\Xi\nn}{\Xi\nn+t^2}\dd t
\lito0.
\end{align}

Let $R\nm:=\bigabs{\DD_{n}\cap[1,m]}$ be the number of roots that are
descendants of $n$.
When the procedure in \refS{Sbasic} reaches $m$, there are $Y_m$ edges
left. Each of these selects an endpoint in \set{1,\dots,m} at random, 
uniformly and independently, and $R\nm$ is the number of vertices in
\set{1,\dots,m} that are selected at least once. 
(This is a classical occupancy problem, often described as throwing $Y_m$
balls into $m$ cells.)

Conditioned on $\cF_m$, 
each vertex $k\le m$ thus has the same probability
$\E J_k=1-(1-\frac{1}{m})^{Y_m}$ of becoming red.
The covariances can easily be calculated, but we note instead that if we
also condition on $J_k=0$, this increases the probability that $J_\ell=1$
for every $\ell\neq k$; thus $\Cov\bigpar{J_k,J_\ell\mid\cF_m}\le0$,
and
\begin{align}\label{ask4}
  \Var \bigpar{R\nm\mid\cF_m}&
 = \sum_{k,\ell=1}^m\Cov\bigpar{J_k,J_\ell\mid\cF_m}
\le \sum_{k=1}^m \Var \bigpar{J_k\mid\cF_m}
= m \Var \bigpar{J_m\mid\cF_m}
\notag\\&
\le m \E \bigpar{J_m\mid\cF_m} \le m \E\bigpar{ Z_m\mid\cF_m} = m\frac{Y_m}{m}=Y_m.
\end{align}
Hence, recalling \eqref{tw3} and \eqref{tm4},
\begin{align}\label{ask5}
  \E \lrpar{\frac{R\nm-\E(R\nm\mid\cF_m)}{\sqrt n}}^2
=\frac{\E \sqpar{  \Var \bigpar{R\nm\mid\cF_m}}}{n}
\le \frac{\E Y_m}{n}
\le \frac{\E W_m}{mn}
\le\frac{2}{m}\to0.
\end{align}
Consequently,
\begin{align}\label{ask6}
  \frac{R\nm-\E(R\nm\mid\cF_m)}{\sqrt n}\pto0.
\end{align}
Furthermore, by symmetry and \eqref{sel3},
\begin{align}\label{ask7}
  \E\bigpar{R\nm\mid\cF_m}
= m\E\bigpar{J_m\mid\cF_m}
= m \frac{\Xi\nn}{\Xi\nn+m^2/n}+\op(m),
\end{align}
where $\op(m)$ is a (random) quantity such that $\op(m)/m\pto0$.
It follows from \eqref{ask6} and \eqref{ask7} that
\begin{align}\label{ask8}
  \frac{R\nm}{\sqrt n} - \mu\frac{\Xi\nn}{\Xi\nn+\mu^2}\pto0.
\end{align}
 
We obtain \eqref{tm2b} by summing \eqref{ask3} and \eqref{ask8},
and this implies \eqref{tm2a} by \eqref{lxi}.
Finally, moment convergence follows, since every power is uniformly
integrable by $|\DD\nm|\le X\nn$ and \refL{LXp}.
\end{proof}

It is possible to obtain results also for the case $m/\sqrt n\to\infty$ by
our methods, but we leave this case to the reader.

\subsection{Drawing without replacement}\label{SSwithout}
Consider now the case when the endpoints of the $d$ edges from a vertex $k$
are selected by drawing without replacement; 
in other words, the endpoints form a uniformly random subset 
of \set{1,\dots,k-1}
with $d$ elements. (We start with $m\ge d$ roots.)
Thus there are no multiple edges and $D_n$ is a simple multigraph.

The analysis in \refS{Sbasic} is based on the independence of the endpoints of
different edges; this is no longer true since edges from the same vertex now
are dependent. However, a minor variation of the arguments allows us to
reach  the same conclusions.
For simplicity, we consider again the case  $d=2$, and leave the
straightforward generalization to higher $d$ to the reader.
We use the same notations as above, with the additions below.

Say that the two edges starting together from a red vertex are \emph{twins}.
We thus now do not allow two twins to have the same endpoint.

Consider the $Y_k$  red edges that cross the gap between $k+1$ and $k$. 
Some of these come in pairs of twins, while others are single
(because their twin has already found an endpoint).
Let $Y_{k,1}$ be the number of single edges, and $Y_{k,2}$ the number of
pairs of twins among these edges. Thus
\begin{align}\label{y12}
  Y_k=Y_{k,1}+2Y_{k,2}.
\end{align}
Similarly, let $Z_{k,1}$ be the number of single edges that end at $k$, and
let $Z_{k,2}$ be the number of edges that end in $k$ and
still having a living twin (that will later find an endpoint $\ell<k$).
Thus 
\begin{align}\label{z12}
  Z_k=Z_{k,1}+Z_{k,2}.
\end{align}
We still have  \eqref{tw1}, but also the more detailed recursion
\begin{align}\label{w1}
  Y_{k-1,1}&=Y_{k,1}-Z_{k,1}+Z_{k,2},
\\\label{w2}
  Y_{k-1,2}&=Y_{k,2}-Z_{k,2}+J_k
=Y_{k,2}-Z_{k,2}+\indic{Z_{k,1}+Z_{k,2}\ge1}.
\end{align}
Each of the $Y_{k,1}$  single edges ends at $k$ with probability $1/k$,
and each of the $Y_{k,2}$ pairs of twins has one edge ending at $k$ (and
thus leaving one single edge) with probability $2/k$.
Hence, conditioned on $\cF_k$, we now have
\begin{align}\label{w3}
  Z_{k,1}&\in\Bin(Y_{k,1},1/k),
\\\label{w4}
  Z_{k,2}&\in\Bin(Y_{k,2},2/k),
\end{align}
with $Z_{k,1}$ and $Z_{k,2}$ (conditionally) independent.

Taking conditional expectations yields,  instead of \eqref{tw2},
\begin{align}\label{w5}
\E \bigpar{ Y_{k-1,1}\mid\cF_k}&= Y_{k,1}-\frac1k Y_{k,1}+\frac2k Y_{k,2},
\\\label{w6}
 \E \bigpar{Y_{k-1,2}\mid\cF_k}&=Y_{k,2}-\frac2k Y_{k,2}+
\P\bigpar{Z_k\ge1\mid\cF_k}
%\\
%&=Y_{k,2}-\frac2k Y_{k,2}+
%2\Bigpar{1-\Bigpar{1-\frac{1}{k}}^{Y_{k,1}}\Bigpar{1-\frac{2}{k}}^{Y_{k,2}}}
.\end{align}
 Thus, using \eqref{y12},
\begin{align}\label{w7}
\E \bigpar{ Y_{k-1}\mid\cF_k}&
= Y_{k}-\frac1k Y_{k}
%+2\Bigpar{1-\Bigpar{1-\frac{1}{k}}^{Y_{k,1}}\Bigpar{1-\frac{2}{k}}^{Y_{k,2}}}
+2\P\bigpar{Z_k\ge1\mid\cF_k}
\notag\\
&= Y_{k}-\frac1k Y_{k}
+2\Bigpar{1-\Bigpar{1-\frac{1}{k}}^{Y_{k,1}}\Bigpar{1-\frac{2}{k}}^{Y_{k,2}}}
\end{align}
and also, exactly as in \eqref{tw2+},
\begin{align}
\E \bigpar{ Y_{k-1}\mid\cF_k}&= Y_{k}-\frac1k Y_{k}
%+2\Bigpar{1-\Bigpar{1-\frac{1}{k}}^{Y_{k,1}}\Bigpar{1-\frac{2}{k}}^{Y_{k,2}}}
+2\P\bigpar{Z_k\ge1}
%\\
\le \frac{k-1}{k}Y_k + 2 \E Z_k
= \frac{k+1}k Y_{k}.
\end{align}
Thus $W_k$ is still a reverse supermartingale, $M_k$ is a revese
martingale, and $A_k$ is reverse increasing; \eqref{tw3}--\eqref{tm4} hold
without any changes. 
The exact formulas in \eqref{tw4a} and \eqref{tw4b} are replaced by
\begin{align}\label{tw4a*}
    \E \bigpar{W_{k-1}\mid\cF_k}&
%=k  \E \bigpar{Y_{k-1}\mid\cF_k}
%=(k-1)Y_k+2k\P\bigpar{Z_{k-1}\ge1\mid\cF_k}
=(k-1)Y_k
+2k\Bigpar{1-\Bigpar{1-\frac{1}{k}}^{Y_{k,1}}\Bigpar{1-\frac{2}{k}}^{Y_{k,2}}}
\end{align}
and thus
\begin{align}\label{tw4b*}
A_{k-1}-A_k&=
W_k- \E \bigpar{W_{k-1}\mid\cF_k}
%\notag\\&
%=2k\bigsqpar{\E\bigpar{Z_{k-1}\mid\cF_k}-\P\bigpar{Z_{k-1}\ge1\mid\cF_k}}
%\notag\\& 
=2k\Bigpar{\Bigpar{1-\frac{1}{k}}^{Y_{k,1}}\Bigpar{1-\frac{2}{k}}^{Y_{k,2}}-1
+\frac{Y_k}{k}}
.\end{align}
The rest of \refS{Sbasic} holds with minor changes: 
the numerical constants in inequalities may change (perhaps including cases
where we had constant 1), 
we estimate (conditional) variances of $Z_{k,1}$ and $Z_{k,2}$ separately in 
\eqref{tw6},
the exact formula in
\eqref{la1} is modified as above, and the equality in \eqref{l24} is
modified; we omit the details.

In \refS{SI}, we note that for the version
studied in the previous sections, the probability that two twins starting at
$k$ have the same endpoint is $1/(k-1)$. Hence, the expected number of such
collisions among twins starting at $k\ge n_1$ is (with $J_n:=1$), 
using \eqref{l21},
\begin{align}
  \sum_{k=n_1}^n\P(J_k=1)\frac{1}{k-1}
\le \frac{1}{n-1}+   \sum_{k=n_1}^{n-1}\frac{2n}{k^2(k-1)}
\le \frac{1}{n-1}+ \frac{3n}{n_1^2}=o(1).
\end{align}
Thus, \whp{} there are no such collisions, which means that we may couple
the versions using drawing with and without replacement such that they
\whp{} coincide on the interval $[n_1,n]$. Consequently,
\refT{TYule} giving a coupling with the Yule process holds also for  drawing
without replacement.

The results in \refSs{SI}--\ref{Smom} now hold as before, with some numerical
constants changed and a few minor changes. 
The most important is that \eqref{kb6a} now, by \eqref{tw4b*}, becomes 
\begin{align}\label{kb6a*}
  \ddt \hA\nn_t
= -n\qqw\cdot 
2k\Bigpar{\Bigpar{1-\frac{1}{k}}^{Y_{k,1}}\Bigpar{1-\frac{2}{k}}^{Y_{k,2}}-1
+\frac{Y_k}{k}}
,\end{align}
but it is easily seen that this together with \eqref{kb5} 
still yields \eqref{kb6},
since
\begin{align}\label{london}
  \log\biggpar{\Bigpar{1-\frac{1}{k}}^{Y_{k,1}}\Bigpar{1-\frac{2}{k}}^{Y_{k,2}}}
&= -Y_{k,1}\cdot\frac{1}{k} - Y_{k,2}\cdot\frac{2}{k} +
  O\Bigpar{Y_k\cdot\frac{1}{k^2}}
\notag\\&
=-\frac{Y_k}{k}+o(1)
.\end{align}
There is a similar modification in \eqref{vs2}, but again the conclusion
\eqref{vs3} holds by \eqref{london}.
In \refS{Smom}, we argue as in \eqref{sw4} for $Z_{k,1}$ and $Z_{k,2}$ separately.

Hence, \refT{TX}  holds also for drawing without replacement.
(And so does \refT{TD}, by similar arguments.)

\section*{Acknowledgement}
I thank Donald Knuth and Philippe Jacquet for drawing my attention to this
problem.

\newcommand\AAP{\emph{Adv. Appl. Probab.} }
\newcommand\JAP{\emph{J. Appl. Probab.} }
\newcommand\JAMS{\emph{J. \AMS} }
\newcommand\MAMS{\emph{Memoirs \AMS} }
\newcommand\PAMS{\emph{Proc. \AMS} }
\newcommand\TAMS{\emph{Trans. \AMS} }
\newcommand\AnnMS{\emph{Ann. Math. Statist.} }
\newcommand\AnnPr{\emph{Ann. Probab.} }
\newcommand\CPC{\emph{Combin. Probab. Comput.} }
\newcommand\JMAA{\emph{J. Math. Anal. Appl.} }
\newcommand\RSA{\emph{Random Structures Algorithms} }
\newcommand\DMTCS{\jour{Discr. Math. Theor. Comput. Sci.} }

\newcommand\AMS{Amer. Math. Soc.}
\newcommand\Springer{Springer-Verlag}
\newcommand\Wiley{Wiley}

\newcommand\vol{\textbf}
\newcommand\jour{\emph}
\newcommand\book{\emph}
\newcommand\inbook{\emph}
\def\no#1#2,{\unskip#2, no. #1,} %(typeset after year) 
\newcommand\toappear{\unskip, to appear}

\newcommand\arxiv[1]{\texttt{arXiv}:#1}
\newcommand\arXiv{\arxiv}

\newcommand\xand{and }
\renewcommand\xand{\& }

\def\nobibitem#1\par{}

\end{document}